\documentclass{article}%
\usepackage{amsfonts, amsmath, amssymb}
\usepackage{subfig}
\usepackage{pstricks, pst-plot, pst-node, multido}
\usepackage{amsmath}
\usepackage{amsfonts}
\usepackage{amssymb}
\usepackage{graphicx}

\newtheorem{theorem}{Theorem}
\newtheorem{acknowledgement}[theorem]{Acknowledgement}

\newtheorem{definition}[theorem]{Definition}

\newtheorem{lemma}[theorem]{Lemma}

\newtheorem{proposition}[theorem]{Proposition}
\newtheorem{remark}[theorem]{Remark}

\newenvironment{proof}[1][Proof]{\noindent\textbf{#1.} }{\ \rule{0.5em}{0.5em}}
\begin{document}

\title{Global bifurcation of polygonal relative equilibria for
masses, vortices and dNLS oscillators}
\author{C. Garc\'{\i}a-Azpeitia
\and J. Ize\\{\small Depto. Matem\'{a}ticas y Mec\'{a}nica, IIMAS-UNAM, FENOMEC, }\\{\small Apdo. Postal 20-726, 01000 M\'{e}xico D.F. }\\{\small cgazpe@hotmail.com}}
\maketitle

\begin{abstract}

Given a regular polygonal arrangement of identical objects, turning around a
central object (masses, vortices or dNLS oscillators), this paper studies the
global bifurcation of relative equilibria in function of a natural parameter
(central mass, central circulation or amplitude of the oscillation). The
symmetries of the problem are used in order to find the irreducible
representations, the linearization and, with the help of a degree theory, the
symmetries of the bifurcated solutions.

Keywords: relative equilibrium,  (n+1)-body, (n+1)-vortex,
dNLS, global bifurcation, degree theory.

MSC 34C25, 37G40, 47H11, 54F45

\end{abstract}

\section{Introduction}

Consider a polygonal arrangement of $n$ identical objects turning in a plane,
at constant angular speed, around a central object. These objects may be
masses, following Newton's law of attraction, or point vortices, with Kirchoff's
law, or nonlinear oscillators coupled to nearest neighbors in a finite circular
lattice and a common phase.

A relative equilibrium for these problems is a stationary solution in rotating
coordinates, \cite{MeHa91}. For each angular speed there is a regular
polygonal relative equilibrium and an associated central quantity (mass, or
circulation or amplitude of the oscillation) which is taken as a parameter.

The purpose of this paper is to prove that for certain explicit values of this
parameter there is a global bifurcation of relative equilibria with a specific
symmetry. The tools for this study is representation theory and a simple
version of the equivariant topological degree, studied in \cite{IzVi03}. As a
matter of fact, the reduction to irreducible representations gives a very
clear picture of the symmetries involved and will be used in forthcoming
papers on the bifurcation of periodic solutions for these problems.

There is a vast literature on the $(n+1)$-body problem, beginning with
Maxwell's model for the Saturn rings. The point vortices problem has also
attracted a good deal of research. This is usually done with a combination of
numerical and explicit computations, where the symmetry is regarded as a
nuisance. In the present paper we hope to show that these symmetries, when
considered as a whole, facilitate instead the study.

In the rest of this introduction we shall set more precisely the problems.
Then, in Section two, we shall see how the symmetry of the problem forces the
Hessian of the system of equations to have a special structure and we shall
introduce a transformation which will bring the Hessian in a block-diagonal
form, according to the different isotropy types. Afterwards, in Section three,
we shall state our bifurcation results, both local and global, giving
solutions with specific symmetries. In Section four, we shall complete the
spectral analysis for the $n$-body and $n$-vortex problems. The final section
is on the discrete NLS, which belongs to a somewhat different field of
application but where a very similar analysis may be performed.

Among the papers listed in the bibliography, in particular \cite{MeSc88},
\cite{Sc03}, {\cite{CaSc00}}, \cite{Ne01} and \cite{PaDo09} and their
respective references, the paper closest to our results is \cite{MeSc88} for
the $n$-body and the $n$-vortex problems. These authors find the same critical
values of the parameters and use a normal form analysis and numerical
computations in order to prove local bifurcation results.

\subsection{$(n+1)$-vortex problem}

One of our purposes is to study relative equilibria of $n+1$ vortices with
circulations $\mu_{0}=\mu$ and $\mu_{j}=1$ for $j\in\{1,...,n\}$. Let
$q_{j}(t)\in\mathbb{R}^{2}$ be the position of the vortices and $x_{j}(t) =
e^{-\omega Jt} q_{j}(t)$ be their position in rotating coordinates, with
angular speed $\omega$. Then, the dimensionless equations for relative
equilibria with frequency $\omega$ are%
\[
\omega\mu_{j}x_{j}=\sum_{i=0(i\neq j)}^{n}\mu_{i}\mu_{j}\frac{x_{j}-x_{i}%
}{\left\Vert x_{j}-x_{i}\right\Vert ^{2}}\text{,}%
\]
where $J$ is the canonical symplectic matrix.

\subsection{$(n+1)$-body problem}

Another of our purposes is to study relative equilibria of $n+1$ bodies in the
plane, where the bodies have masses $\mu_{0}=\mu$ and $\mu_{j}=1$ for
$j\in\{1,..,n\}$. Let $x_{j}(t)$ be the position of the bodies in the plane
and in rotating coordinates, as above with angular speed $(\omega)^{1/2}$. It
is well known that, after the change of variables, the dimensionless equations
for relative equilibria with frequency $(\omega)^{1/2}$ are%
\[
\omega\mu_{j}x_{j}=\sum_{i=0(i\neq j)}^{n}\mu_{i}\mu_{j}\frac{x_{j}-x_{i}%
}{\left\Vert x_{j}-x_{i}\right\Vert ^{3}}\text{.}%
\]

\subsection{General problem}

Now, we will set a formulation generalizing both previous problems. Let $x$ be
the vector $(x_{0},x_{1},...,x_{n})^T$, where $T$ denotes the transposed,
and $\mathcal{M}$ the matrix
$diag(\mu,1,...,1)$. Our aim is to look for critical points of the potential%
\begin{equation}
V_{\alpha}(x)=\omega(x^T\mathcal{M}x)/2+\sum_{i<j}%
\mu_{i}\mu_{j}\phi_{\alpha}(\left\Vert x_{j}-x_{i}\right\Vert )\text{,}
\label{Eq13}%
\end{equation}
where $\phi_{\alpha}(x)$ satisfies $\phi_{\alpha}^{\prime}(x)=-1/x^{\alpha}$
for $\alpha\in\lbrack1,\infty)$.

Since the potential $V$ has gradient%
\[
\nabla_{x_{j}}V(x)=\omega\mu_{j}x_{j}-\sum_{i=0~(i\neq j)}^{n}\mu_{i}\mu
_{j}\frac{x_{j}-x_{i}}{\left\Vert x_{j}-x_{i}\right\Vert ^{\alpha+1}}\text{,}%
\]
then the critical points of $V$ are the relative equilibria of the vortex
problem for $\alpha=1$, and of the body problem for $\alpha=2$. Also, the case
$\alpha\geq1$ can be regarded as a problem of relative equilibria for bodies,
where the general attraction potential is $\phi_{\alpha}$.

Hereafter we represent points in $\mathbb{R}^{2}$ and $\mathbb{C}$
indistinctly. Let $\zeta=2\pi/n$ and let us set the positions of the bodies
$a_{j}$ as $a_{0}=0$ and $a_{j}=e^{ij\zeta}$ for $j\in\{1,...,n\}$. We see
that $a_{j}$ form a relative equilibrium with a central massive body at the
origin surrounded by bodies of equal masses in a regular polygon. This
polygonal relative equilibrium was studied by Maxwell as a simplified model
of Saturn and its rings.

\begin{proposition}
$\bar{a}=(a_{0},...,a_{n})$ is a critical point of the potential $V(x)$, when
$\omega=\mu+s_{1}$ with%
\[
s_{1}=\sum_{j=1}^{n-1}\frac{1-e^{ij\zeta}}{\left\Vert 1-e^{ij\zeta}\right\Vert
^{\alpha+1}}=\frac{1}{2^{\alpha}}\sum_{j=1}^{n-1}\frac{1}{\sin^{(\alpha
-1)}(j\zeta/2)}.
\]

\end{proposition}

\begin{proof}
For $j=0$, we have $\nabla_{x_{0}}V(\bar{a})=\mu\sum_{j=0}^{n-1}e^{ij\zeta}=0
$. For $j\neq0$, we have%
\[
\nabla_{x_{j}}V(\bar{a})=\omega a_{j}-\sum_{i=1~(i\neq j)}^{n}\frac
{a_{j}-a_{i}}{\left\Vert a_{j}-a_{i}\right\Vert ^{\alpha+1}}-\mu a_{j}%
=a_{j}\left( \omega-(\mu+s_{1})\right) ,
\]
Therefore, $\bar{a}$ is a relative equilibrium for frequencies $\omega
=\mu+s_{1}$.
\end{proof}

Notice that any homograph of a relative equilibrium is also a relative
equilibrium. This is why we have decided to fix the norm of the relative
equilibrium and leave the parameter $\mu$ free. Our objective is to find
global bifurcation of relative equilibria from $a_{j}$ using the parameter
$\mu$. Now let us see the symmetries of the problem.

\begin{definition}
\label{EnAc}Let $S_{n}$ be the group of permutations of $\{1,...,n\}$ and let
$D_{n}$ be the subgroup generated by the permutations $\zeta(j)=j+1$ and
$\kappa(j)=n-j$. We define the action of $S_{n}$ in $\mathbb{R}^{2(n+1)} $ as%
\[
\rho(\gamma)(x_{0},x_{1},...,x_{n})=(x_{0},x_{\gamma(1)},...,x_{\gamma(n)}).
\]
In addition, we define the action of $O(2)=S^{1}\cup\kappa S^{1}$ as%
\[
\rho(\theta)=e^{-\mathcal{J\theta}}\text{ and }\rho(\kappa)=\mathcal{R}%
\text{,}%
\]
where $\mathcal{R}$ is the matrix $diag(R,...,R)$ with $R=diag(1,-1)$.
\end{definition}

Because $n$ of the bodies have equal masses, the potential $V$ is $S_{n}%
$-invariant. Moreover, the potential $V$ is $O(2)$-invariant since the
equations are invariant by rotating or reflecting the positions of all the
bodies. Consequently, the gradient $\nabla V$ is $\Gamma$-equivariant with
$\Gamma=S_{n}\times O(2)$. This means just that%
\[
\nabla V(\rho(\gamma)x)=\rho(\gamma)\nabla V(x)
\]
for all $\gamma\in\Gamma$.

Let $\tilde{D}_{n}$ be the group generated by the elements $(\zeta,\zeta)$ and
$(\kappa,\kappa)$ of $S_{n}\times O(2)$, where $\zeta=2\pi/n\in S^{1} $. The
action of $(\zeta,\zeta)$ and $(\kappa,\kappa)$ in $\mathbb{R}^{2(n+1)}$ is%
\[
(\zeta,\zeta)x=\rho(\zeta)e^{-\mathcal{J}\zeta}x\text{ and }(\kappa
,\kappa)x=\rho(\kappa)\mathcal{R}x\text{.}%
\]
As the action of $(\zeta,\zeta)$ and $(\kappa,\kappa)$ leaves the equilibrium
$\bar{a}$ fixed, then its isotropy group, i.e. the subgroup of $\Gamma$ which
fixes the orbit $\bar{a}$, is%
\[
\Gamma_{\bar{a}}=\tilde{D}_{n}.
\]

\section{Irreducible representations}

In order to prove the bifurcation theorem, we need to find the spaces of
irreducible representations of $\tilde{D}_{n}$.

Let us define $A_{ij}$ to be the $2\times2$ submatrices of $D^{2}V(\bar{a})$
such that%
\[
D^{2}V(\bar{a})=A=(A_{ij})_{ij=0}^{n}\text{.}%
\]
Due to the fact that $D^{2}V(\bar{a})$ is $\tilde{D}_{n}$-equivariant, one has
the following result:

\begin{proposition}
The blocks $A_{ij}$ satisfy the relations%
\begin{equation}
A_{ij}=e^{-J\zeta}A_{\zeta(i)\zeta(j)}e^{J\zeta}\text{ and }A_{ij}%
=RA_{\kappa(i)\kappa(j)}R\text{.} \label{EqSy}%
\end{equation}

\end{proposition}

\begin{proof}
Since the matrix $D^{2}V(\bar{a})$ is $\tilde{D}_{n}$-equivariant, then the
matrix $A$ and $\rho(\zeta)e^{-\mathcal{J\zeta}}$ commute. Therefore,%
\[
A=\rho(\zeta)e^{-\mathcal{J\zeta}}Ae^{\mathcal{J\zeta}}\rho(\zeta
)^{-1}\text{.}%
\]
Hereafter, we denote by $[u]_{i}$ the coordinate $u_{i}\in\mathbb{R}^{2}$ of
the vector $u=(u_{0},...,u_{n})^T$. Therefore,%
\begin{align*}
\lbrack\rho(\zeta)e^{-\mathcal{J\zeta}}Ae^{\mathcal{J\zeta}}\rho(\zeta
)^{-1}u]_{i} & =e^{-J\zeta}[Ae^{\mathcal{J}\zeta}\rho(\zeta)^{-1}%
u]_{\zeta(i)}\\
& =e^{-J\zeta}\sum A_{\zeta(i)j}(e^{J\zeta})u_{\zeta^{-1}(j)}=\sum
e^{-J\zeta}A_{\zeta(i)\zeta(j)}e^{J\zeta}u_{j}\text{.}%
\end{align*}
From this equality, we get that%
\[
\sum_{j}A_{ij}u_{j}=[Au]_{i}=\sum e^{-J\zeta}A_{\zeta(i)\zeta(j)}e^{J\zeta
}u_{j}\text{.}%
\]
Then we conclude that $A_{ij}=e^{-J\zeta}A_{\zeta(i)\zeta(j)}e^{J\zeta}$.
Using a similar argument and the fact that $A$ and $\rho(\kappa)\mathcal{R}$
commute, we obtain $A_{ij}=RA_{\kappa(i)\kappa(j)}R$.
\end{proof}

Now, we may find the irreducible representations of the action of the group
$\tilde{D}_{n}$.

Since the irreducible representations are different for $n=2$ and $n\geq3$, we
shall concentrate of the case $n\geq3$ in the remaining part of the paper, except
for comments on the case $n=2$.

\begin{definition}
\label{EnVk}Let us define the vectors $v_{1}$ and $v_{n-1}$ as%
\[
v_{1}=\frac{1}{\sqrt{2}}\left(
\begin{array}
[c]{c}%
1\\
i
\end{array}
\right) \text{ and }v_{n-1}=\frac{1}{\sqrt{2}}\left(
\begin{array}
[c]{c}%
1\\
-i
\end{array}
\right) \text{.}%
\]
For $k\in\{2,...,n-2,n\}$, we define the isomorphisms $T_{k}:\mathbb{C}%
^{2}\rightarrow V_{k}$ as%
\begin{align*}
T_{k}(z) & =(0,n^{-1/2}e^{(ikI+J)\zeta}z,...,n^{-1/2}e^{n(ikI+J)\zeta
}z)\text{ with}\\
V_{k} & =\{(0,e^{(ikI+J)\zeta}z,...,e^{n(ikI+J)\zeta}z):z\in\mathbb{C}%
^{2}\}\text{,}%
\end{align*}
and for $k\in\{1,n-1\}$, we define the isomorphism $T_{k}:$ $\mathbb{C}%
^{3}\rightarrow V_{k}$ as
\begin{align*}
T_{k}(\alpha,w) & =(v_{k}\alpha,n^{-1/2}e^{(ikI+J)\zeta}w,...,n^{-1/2}%
e^{n(ikI+J)\zeta}w)\text{ with}\\
V_{k} & =\{(v_{k}\alpha,e^{(ikI+J)\zeta}w,...,e^{n(ikI+J)\zeta}%
w):w\in\mathbb{C}^{2},\alpha\in\mathbb{R}\}.
\end{align*}

\end{definition}

Next let us find the action of the group $\tilde{D}_{n}$ on the subspaces
$V_{k}$.

\begin{proposition}
\label{EnNeAc}The actions of $(\zeta,\zeta)$ and $(\kappa,\kappa)$ on $V_{k}$
are%
\begin{align*}
(\zeta,\zeta)T_{k}(z) & =T_{k}(e^{ik\zeta}z)\text{ and}\\
(\kappa,\kappa)T_{k}(z) & =T_{n-k}(Rz)\text{,}%
\end{align*}
where $R$ is the matrix $diag(1,1,-1)$ for the special cases $k\in\{1,n-1\}$ .
\end{proposition}

\begin{proof}
For $k\in\{2,..,n-2,n\}$, we have
\begin{align*}
\lbrack\rho(\zeta)T_{k}(z)]_{j} & =[T_{k}(z)]_{\zeta(j)}=n^{-1/2}%
e^{j(ikI+J)\zeta}(e^{(ikI+J)\zeta}z)\\
& =[T_{k}(e^{(ikI+J)\zeta}z)]_{j}\text{.}%
\end{align*}
Therefore $\rho(\zeta)T_{k}(z)=T_{k}(e^{(ikI+J)\zeta}z)$. Since the element
$\zeta\in O(2)$ acts as $e^{-\zeta\mathcal{J}}T_{k}(z)=T_{k}(e^{-\zeta J}z)$,
we conclude that $(\zeta,\zeta)$ acts as $(\zeta,\zeta)T_{k}(z)=T_{k}%
(e^{ik\zeta}z)$.

For $k\in\{1,n-1\}$ we have, as before, that%
\[
\rho(\zeta)T_{k}(\alpha,w)=T_{k}(\alpha,e^{(ikI+J)\zeta}w).
\]
Moreover, from the equality $e^{-J\zeta}v_{k}=e^{ik\zeta}v_{k}$ we get that
$\zeta\in O(2)$ acts as%
\[
e^{-\zeta\mathcal{J}}T_{k}(\alpha,w)=T_{k}(e^{ik\zeta}v_{k}\alpha,e^{-J\zeta
}w).
\]
Hence, the action in this case is also $(\zeta,\zeta)T_{k}(z)=T_{k}%
(e^{ik\zeta}z)$.

It remains to find the action of $(\kappa,\kappa)$. For $k\in\{2,...,n-2,n\}$,
we have%
\begin{align*}
\lbrack(\kappa,\kappa)T_{k}(z)]_{j} & =[\mathcal{R}T_{k}(z)]_{\kappa
(j)}=n^{-1/2}e^{j(-ikI+J)\zeta}Rz\\
& =[T_{n-k}(Rz)]_{j}\text{,}%
\end{align*}
therefore the action is $(\kappa,\kappa)T_{k}(z)=T_{n-k}(Rz)$. For
$k\in\{1,n-1\}$, by a similar argument and the fact that $Rv_{k}=v_{\kappa
(k)}$, we prove that the action is as before but with $R=diag(1,1,-1)$.
\end{proof}

Consequently, we have that the spaces $V_{k}\oplus V_{n-k}$ are
subrepresentations of the action of $\tilde{D}_{n}$. Moreover, the action of
$(\zeta,\zeta)$ and $(\kappa,\kappa)$ on the subspace $V_{k}\oplus V_{n-k}$
is
\begin{align*}
(\zeta,\zeta)(z_{k},z_{n-k}) & =(e^{ik\zeta}z_{k},e^{-ik\zeta}z_{n-k})\text{
and}\\
(\kappa,\kappa)(z_{k},z_{n-k}) & =(Rz_{n-k},Rz_{k})\text{.}%
\end{align*}

Let $\mathbb{\tilde{Z}}_{n}$ be the group generated by $(\zeta,\zeta)$. Since
the subspaces $V_{k}$ are irreducible representations of $\mathbb{\tilde{Z}%
}_{n}$, by Schur's lemma (that is a linear map which commutes with action, must send
equivalent representations into themselves), we obtain
$D^{2}V(\bar{a})T_{k}(z)=T_{k}(B_{k}z)$.
Furthermore, as $D^{2}V(\bar{a})$ commutes with the action of $(\kappa
,\kappa)$, then the blocks $B_{k}$ must satisfy $B_{k}R=RB_{n-k}$.
Consequently, there must be a map that puts the matrix $D^{2}V(\bar{a})$ in
diagonal form with the blocks $B_{k}$. Clearly, the isomorphisms $T_{k}$, with
range $V_{k}$, are the components of this orthogonal transformation.

\begin{proposition}
Define the map $Pz=\sum_{k=1}^{n}T_{k}(z_{k})$ for $z=(z_{1},...,z_{n})$, then
the linear map $P$ is orthogonal $P^{\ast}=P^{-1} $.
\end{proposition}

\begin{proof}
Since the matrix $e^{jJ\zeta}$ is an isometry in $\mathbb{C}^{2}$ and
\[
\sum_{j=0}^{n-1}e^{ij(k-l)\zeta}=n\delta_{kl},
\]
for $k,l\in\{2,...,n-2,n\}$, then%
\[
\left\langle T_{k}(z_{k}),T_{l}(z_{l})\right\rangle =n^{-1}\sum_{j=1}%
^{n}e^{ij(k-l)\zeta}\left\langle e^{jJ\zeta}z_{k},e^{jJ\zeta}z_{l}%
\right\rangle =\delta_{kl}\left\langle z_{k},z_{l}\right\rangle \text{.}%
\]
In fact, since $v_{1}$ and $v_{n-1}$ are orthonormal vectors, one proves that
$\left\langle T_{k}(z_{k}),T_{l}(z_{l})\right\rangle =\delta_{kl}\left\langle
z_{k},z_{l}\right\rangle $ for all $k$ and $l$. Therefore, the map $P$
satisfies
\[
\left\langle Pz,Pz\right\rangle =\sum\delta_{kl}\left\langle z_{k}%
,z_{l}\right\rangle =\left\langle z,z\right\rangle \text{.}%
\]
Thus, $P$ is an isometry on $\mathbb{C}^{2(n+1)}$ and $P^{\ast}=P^{-1}$.
\end{proof}

By Schur's lemma, the matrix $D^{2}V(\bar{a})$ is diagonal in the new
coordinates, that is%
\[
P^{-1}D^{2}V(\bar{a})P=diag(B_{1},...,B_{n})\text{.}%
\]
Our next objective consists in finding the blocks $B_{k}$ in terms of the
matrices $A_{ij}$. Remember that the matrices $A_{ij}$ are the $2\times2$
submatrices of the Hessian $D^{2}V(\bar{a})$.

\begin{proposition}
\label{EnBk}For $k\in\{2,...,n-2,n\}$ the blocks $B_{k}$ are
\[
B_{k}=\sum_{j=1}^{n}A_{nj}e^{j(ikI+J)\zeta}\text{.}%
\]

\end{proposition}

\begin{proof}
For $l\neq0$ we have $[AT_{k}(z)]_{l}=n^{-1/2}\sum_{j=1}^{n}A_{lj}%
e^{j(ikI+J)\zeta}z$. Now, from the relation (\ref{EqSy}), we prove that
$A_{lj}=e^{lJ\zeta}A_{n(j-l)}e^{-lJ\zeta}$, with $l-j\ $ modulo $n$. Hence%
\[
\lbrack AT_{k}(z)]_{l}=n^{-1/2}\sum_{j=1}^{n}e^{l(ik+J)\zeta}A_{n(j-l)}%
e^{(j-l)(ikI+J)\zeta}z\text{.}%
\]
Consequently, we rewrite the sum as%
\[
\lbrack AT_{k}(z)]_{l}=n^{-1/2}e^{l(ik+J)\zeta}\left[ \sum_{j=1}^{n}%
A_{nj}e^{j(ikI+J)\zeta}z\right] =[T_{k}(B_{k}z)]_{l}.
\]
But since the isomorphisms $T_{k}$ are defined on the invariant subspaces
$V_{k}$, then $[AT_{k}(z)]_{0}=[T_{k}(B_{k}z)]_{0}$ and we conclude that
$AT_{k}(z)=T_{k}(B_{k}z)$. Actually, one may prove directly that
$[AT_{k}(z)]_{0}=[T_{k}(B_{k}z)]_{0}$, for instance see \cite{Ga10}.
\end{proof}

\begin{proposition}
For $k\in\{1,n-1\}$ the blocks $B_{k}$ are%
\[
B_{k}=\left(
\begin{array}
[c]{cc}%
e_{1}^{T}A_{00}e_{1} & n^{1/2}\overline{(A_{n0}v_{k})}^{T}\\
n^{1/2}A_{n0}v_{k} & \sum_{j=1}^{n}A_{nj}e^{j(ikI+J)\zeta}%
\end{array}
\right) \text{.}%
\]

\end{proposition}

\begin{proof}
For $l\neq0$, we have%
\[
\lbrack AT_{k}(\alpha,w)]_{l}=(A_{l0}v_{k})\alpha+n^{-1/2}\sum_{j=1}^{n}%
A_{lj}e^{j(ikI+J)\zeta}w\text{,}%
\]
where $\alpha\in\mathbb{C}$ and $w\in\mathbb{C}^{2}$. Now, from the symmetries
(\ref{EqSy}) we prove that $A_{l0}=e^{lJ\zeta}A_{n0}e^{-lJ\zeta}$. Hence,
$A_{l0}v_{k}=e^{lJ\zeta}A_{n0}e^{-lJ\zeta}v_{k}$. Moreover, since
$e^{-lJ\zeta}v_{k}=e^{lik\zeta}v_{k}$, then $A_{l0}v_{k}=e^{l(ki+J)\zeta
}A_{n0}v_{k}$. Using the previous computation, we find that%
\begin{equation}
\lbrack AT_{k}(\alpha,w)]_{l}=n^{-1/2}e^{l(ikI+J)\zeta}\left[ \left(
n^{1/2}A_{n0}v_{k}\right) \alpha+\left( \sum_{j=1}^{n}A_{nj}e^{j(ikI+J)\zeta
}\right) w\right] \text{.} \label{Eq01}%
\end{equation}

For $l=0$, we have $[AT_{k}(z)]_{0}=\left( A_{00}v_{k}\right) \alpha
+n^{-1/2}D_{k}w$, with $D_{k}=\sum_{j=1}^{n}A_{0j}e^{j(ikI+J)\zeta}$. From the
relations (\ref{EqSy}), we have that $A_{00}=cI$. Now, since $v_{k}\bar{v}%
_{k}^{T}=1$, then
\begin{equation}
\lbrack AT_{k}(\alpha,w)]_{0}=v_{k}\left[ \left( e_{1}^{T}A_{00}%
e_{1}\right) \alpha+n^{-1/2}(\bar{v}_{k}^{T}D_{k})w\right] \text{.}
\label{Eq02}%
\end{equation}

Consequently, from the equalities (\ref{Eq01}) and (\ref{Eq02}), we obtain
$D^{2}V(\bar{a})T_{k}(\alpha,w)=T_{k}(B_{k}(\alpha,w))$, with
\[
B_{k}=\left(
\begin{array}
[c]{cc}%
e_{1}^{T}A_{00}e_{1} & n^{-1/2}(\bar{v}_{k}^{T}D_{k})\\
n^{1/2}A_{n0}v_{k} & \sum_{j=1}^{n}A_{nj}e^{j(ikI+J)\zeta}%
\end{array}
\right) .
\]
Moreover, since the map $P$ is orthonormal and the matrix $A$ is selfadjoint,
then $B_{k}$ must be selfadjoint and $\overline{(\bar{v}_{k}^{T}D_{k})}%
^{T}=nA_{n0}v_{k}$. Actually, one may prove directly that $nA_{n0}%
v_{k}=\overline{(\bar{v}_{k}^{T}D_{k})}^{T}=D_{k}^{T}v_{k}$, for instance see
\cite{Ga10}.
\end{proof}

\begin{remark}
In the computation of the blocks $B_{k}$, we have used only the symmetries of
$D^{2}V(\bar{a})$. This will enable us to apply these results to a wide class
of problems, as the dNLS equations at the final section. Also, notice that the
change of variables was done in complex coordinates, and these will allow us
to prove bifurcation of periodic solutions in a series of forthcoming papers
analogous to \cite{GaIz10}: as a matter of fact, the natural approach to the
study of periodic solutions is, in this context, the use of Fourier series. Thus,
the change of variables, which we have introduced, will be helpful.

\end{remark}

However, in order to find bifurcation of relative equilibria, we need the
change of variables for real coordinates.

\begin{proposition}
If the matrix $D^{2}V(\bar{a})$\ has domain $\mathbb{R}^{2(n+1)}$, then the
matrix $P^{-1}D^{2}V(\bar{a})P\ $has domain $W=P^{-1}\mathbb{R}^{2(n+1)}$ and%
\begin{align*}
P^{-1}D^{2}V(\bar{a})P & =diag\left( B_{1},B_{2},...,B_{n/2},B_{n}\right)
\text{ with}\\
W & =\mathbb{C}^{3}\times\mathbb{C}^{2}\times...\times\mathbb{R}^{2}%
\times\mathbb{R}^{2}\text{.}%
\end{align*}
Moreover, the action on the block $B_{k}$ is%
\[
(\zeta,\zeta)z_{k}=e^{ik\zeta}z_{k}\text{ and }(\kappa,\kappa)z_{k}=R\bar
{z}_{k},
\]
where $R$ is the matrix $diag(1,1,-1)$ when $k=1$ and $diag(1,-1)$ for the
remaining cases.
\end{proposition}

\begin{proof}
First, we need to identify the subspace $W=\{z:Pz\in\mathbb{R}^{2(n+1)}\}$. If
$Pz$ is real, then
\[
\sum_{k=1}^{n}T_{k}(z_{k})=Pz=\overline{Pz}=\sum_{k=1}^{n}T_{n-k}(\bar{z}%
_{k})\text{.}%
\]
Thus, the subspace $W$ is the set of points $(z_{1},...,z_{n})$ such that
$z_{n-k}=\bar{z}_{k}$.

Remember that $(\kappa,\kappa)$ acts on the coordinate $z_{k}$ as
$(\kappa,\kappa)z_{k}=Rz_{n-k}$. Hence, for $k\in\{n/2,n\}$ we have
$z_{k}=z_{n-k}=\bar{z}_{k}\in\mathbb{R}^{2}$, then $(\kappa,\kappa)$ acts as
$(\kappa,\kappa)z_{k}=Rz_{k}$. Consequently, the blocks $B_{n/2}$ and $B_{n}$
are defined in a real space with real action.

Now for $k\notin\{n/2,n\}$, we can take the isomorphism $T(z_{k}%
)=(z_{k},z_{n-k})$ with $z_{n-k}=\bar{z}_{k}$. In this way $(\kappa,\kappa)$
acts as%
\[
(\kappa,\kappa)T(z_{k})=(R\bar{z}_{k},Rz_{k})=T(R\bar{z}_{k}).
\]
Moreover, since $(\zeta,\zeta)$ acts as $(\zeta,\zeta)z_{k}=e^{ik\zeta}z_{k}$,
then $(\zeta,\zeta)T(z_{k})=T(e^{ik\zeta}z_{k})$. Finally, we use the equality
$B_{n-k}=\bar{B}_{k}$ to prove that $(B_{k},B_{n-k})T(z_{k})=T(B_{k}z_{k})$.
\end{proof}

\begin{remark}

If $n=2$ we have to define $T_{2}$ as before, from $\mathbb{C}^{2}$ into $V_{2}$.

However, for $k=1$, define the isomorphism $T_{1}:$
$\mathbb{C}^{4}\rightarrow V_{1}$ as%
\begin{align*}
T_{1}(v,w) & =(v,2^{-1/2}w,2^{-1/2}w)\text{ con}\\
V_{1} & =\{(v,w,w):v,w\in\mathbb{C}^{2}\}.
\end{align*}

Then, tha action of $\tilde{D}_{2}$ on $z_{k}\in
V_{k}$ is
\begin{align*}
(\zeta,\zeta)z_{2} & =z_{2}\text{ y }(\kappa,\kappa)z_{2}=Rz_{2}\text{,}\\
(\zeta,\zeta)z_{1} & =-z_{1}\text{ y }(\kappa,\kappa)z_{1}=diag(R,R)z_{1}%
\text{.}%
\end{align*}

Hence, the spaces $V_{k}$ are irreducible for the action of $(\zeta,\zeta)$, but
$V_{1}$ contains two representations, one where $(\kappa,\kappa)$ acts as the
identity and the other where this element acts minus the identity.

One may prove that $A_{ij}$ are diagonal matrices ans satisfy
$A_{11}=A_{22}$, $A_{21}=A_{12}$ and
$A_{01}=A_{10}=A_{02}=A_{20}$. Thus,

\begin{align*}
A_{00} & =(s_{1}+\mu)\mu I-2A_{20}\text{, }A_{20}=-\mu diag(\alpha
,-1)\text{,}\\
A_{21} & =-\frac{1}{2^{\alpha+1}}diag(\alpha,-1)\text{ y }A_{22}=(s_{1}%
+\mu)I-(A_{20}+A_{21})\text{.}%
\end{align*}

In particular, contrary to the case $n\geq3$, where $A_{00}$ is a multiple of the
identity, this matrix is only diagonal.

The transformation $P$ is orthogonal and one sends the hessian into 
$diag(B_1,B_2)$, with $B_{2}=(\alpha+1)(\mu+s_{1})diag(1,0)$, but $B_{1}$ is now
\[
\left(
\begin{array}
[c]{cccc}%
\mu\left( s_{1}+\mu+2\alpha\right) & 0 & -\sqrt{2}\alpha\mu & 0\\
0 & \mu\left( s_{1}+\mu-2\right) & 0 & \sqrt{2}\mu\\
-\sqrt{2}\alpha\mu & 0 & s_{1}+(\alpha+1)\mu & 0\\
0 & \sqrt{2}\mu & 0 & s_{1}%
\end{array}
\right) \text{.}%
\]

In order to study the bifurcation of relative equilibria we need to restrict
the block $B_1$ to the fixed point subspace of $(\kappa,\kappa)$, that is to the
first and third coordonates. There,
\[
B_{1}|_{V_{1}^{(\kappa,\kappa)}}=\left(
\begin{array}
[c]{cc}%
\mu\left( s_{1}+\mu+2\alpha\right) & -\sqrt{2}\alpha\mu\\
-\sqrt{2}\alpha\mu & s_{1}+(\alpha+1)\mu
\end{array}
\right) \text{.}%
\]

The determinant of this matrix is $\mu\left( \mu+s_{1}\right) \left(
2\alpha+\mu+s_{1}+\alpha\mu\right) $. At$\mu=0, -s_{1}$, there will be a
bifurcation, as in the case $n\geq3$, and at
\[
\mu_{1}=-\left( 2\alpha+s_{1}\right) /(\alpha+1)
\]
there will be a bifurcation with symmetry $\tilde{D}_{1}$. For the vortices, one has 
$\alpha=1$, $s_{1}=1/2$ and $\mu_{1}=-5/4$, while, for the masses, one has
$\alpha=2$, $s_{1}=1/4$ and $\mu_{1}=-17/12$.

\end{remark}

\section{Bifurcation theorem}

We shall now give sufficient conditions for the bifurcation of relative
equilibria from $\bar{a}$. To carry on this proof, we need to apply the change
of variables $P$ directly for the potential $V$.

In this manner, we define the potential $V_{P}:W\rightarrow\mathbb{R}$ as
$V_{P}(x)=V(Px)$. Then the potential $V_{P}$ is $\Gamma$-invariant and the
gradient $\nabla V_{P}$ is $\Gamma$-equivariant with the action $\rho
_{P}(\gamma)=P^{-1}\rho(\gamma)P$. Note that $x_{0}=P^{-1}\bar{a}$ is the
relative equilibrium in the new coordinates.

To find the symmetries, for each $h$ dividing $n$, we define the group
$\tilde{D}_{h}$ as the one generated by the elements $(n/h)(\zeta,\zeta)$ and
$(\kappa,\kappa)$. These groups $\tilde{D}_{h}$ are subgroups of the isotropy
group $\tilde{D}_{n}$. Our approach consists in applying Brouwer degree to the
maps $\nabla V_{P}(x)$ restricted to the spaces of fixed points of $\tilde
{D}_{h}$. As seen in $\cite{IzVi03}$, this is equivalent to the $\tilde{D}%
_{n}$-equivariant degree.

So we set the function $f_{h}(x)$ as
\[
f_{h}(x)=\nabla V_{P}(x)|_{W^{\tilde{D}_{h}}}:W^{\tilde{D}_{h}}\rightarrow
W^{\tilde{D}_{h}}.
\]
Then, the zeros of $f_{h}(x)$ are the relative equilibria with symmetry
$\tilde{D}_{h}$. Now, the polygonal relative equilibrium is $x_{0}=P^{-1}%
\bar{a}$, so $x_{0}$ is a zero of $f_{h}(x)$. Since we wish to prove existence
of bifurcation from $x_{0}$, we need the sign of $\det f_{h}^{\prime}(x_{0})$.

\begin{proposition}
\label{EnBiIn} Define $\sigma_{k}$ as%
\begin{align}
\sigma_{k} & =sgn~(e_{1}^{T}B_{k}e_{1})\text{ for }k\in\{n,n/2\}\text{
and}\label{EqSg}\\
\sigma_{k} & =sgn~(\det B_{k})\text{ for }k\in\lbrack1,n/2)\cap
\mathbb{N}.\nonumber
\end{align}
Then
\[
sgn~\left( \det f_{h}^{\prime}(x_{0})\right) =n(\mu)=\sigma_{n}\prod
_{j\in\lbrack1,n/2]\cap\mathbb{N}h}\sigma_{j}.
\]

\end{proposition}

\begin{proof}
Since $(\zeta,\zeta)$ acts on the coordinate $z_{k}$ as $e^{ik\zeta}z_{k}$,
the action of $(n/h)(\zeta,\zeta)$ on $z_{k}$ is trivial when $z_{k}%
=e^{ik(2\pi/h)}z_{k}$. This happens precisely for the coordinates $z_{k} $
with $k\in h\mathbb{N}$. Now, the action of $(\kappa,\kappa)$ on the
coordinate $z_{k}$ is trivial whenever $z_{k}=R\bar{z}_{k}$. Therefore, $z\in
W^{\tilde{D}_{h}}$ only when $z=(z_{h},z_{2h},...,z_{n})$ with $z_{k}=R\bar
{z}_{k}$.

Then, the matrix $D^{2}V_{P}(x_{0})$, on the space $W^{\tilde{D}_{h}}$, is
\[
diag(D_{h},D_{2h},...,D_{n})\text{,}%
\]
where the blocks $D_{h}$ are as follows:

\begin{itemize}
\item For $k\in\{n,n/2\}$, we have that $D_{k}=e_{1}^{T}B_{k}e_{1}$ because
$z_{h}\in\mathbb{R}\times\{0\}$.

\item For $k\notin\{1,n/2,n\}$, since $z_{h}\in\mathbb{R}\times i\mathbb{R} $,
we have that $D_{k}=T^{\ast}B_{k}T$, where $T=diag(1,i)$ is the natural
isomorphism between $\mathbb{R}^{2}$ and $\mathbb{R}\times i\mathbb{R}$.

\item For $k=1$, since $z_{h}\in\mathbb{R}^{2}\times i\mathbb{R}$, we have
$D_{k}=T^{\ast}B_{k}T$, where $T=diag(1,1,i)$ is the natural isomorphism
between $\mathbb{R}^{3}$ and $\mathbb{R}^{2}\times i\mathbb{R}$.
\end{itemize}

Finally, from the definition of $\sigma_{k}$, we get that $sgn~(\det
D_{k})=\sigma_{k}$ and therefore%
\[
sgn~\det f_{p}^{\prime}(x_{0})=sgn~\left( \det D_{n}\prod_{j\in
\lbrack1,n/2]\cap\mathbb{N}h}\det D_{j}\right) =n(\mu)\text{.}%
\]

\end{proof}

Hence, we have given the sign of $\det f_{h}^{\prime}(x_{0})$ in terms of the
blocks $B_{k}$.

\subsection{Local bifurcation}

In order to apply Brouwer degree and prove bifurcation, let us define $f$, from
$B_{2\varepsilon}\times B_{2\rho}$ to $\mathbb{R}\times W$, as%
\begin{align*}
f(x,\mu) & =(\left\Vert x-x_{0}\right\Vert -\varepsilon,f_{h}(x,\mu))\text{
with}\\
B_{2\varepsilon}\times B_{2\rho} & =\text{ }\{(x,\mu)\in W^{\tilde{D}_{h}%
}\times\mathbb{R}:\left\Vert x-x_{0}\right\Vert \leq2\varepsilon,\left\vert
\mu-\mu_{0}\right\vert \leq2\rho\}\text{.}%
\end{align*}

\begin{theorem}
The Brouwer degree of $f(x,\mu)$ is well defined and%
\[
\deg(f;B_{2\varepsilon}\times B_{2\rho})=\eta_{h}(\mu_{0})=n_{h}(\mu_{0}%
-\rho)-n_{h}(\mu_{0}+\rho)\text{.}%
\]
Hence, when $\eta_{h}(\mu_{0})\neq0$, there is a local bifurcation from
$(x_{0},\mu_{0})$ with symmetry $\tilde{D}_{h}$ .
\end{theorem}

\begin{proof}
As in \cite{Iz95}, the proof consists in a linear deformation of the function
$\left\Vert x-x_{0}\right\Vert -\varepsilon$ to $\rho-\left\Vert \mu-\mu
_{0}\right\Vert $ and of the function $f_{h}(x)$ to $f_{h}^{\prime}(x-x_{0})$.
Then, we may use the excision property to prove that%
\[
\deg(f(\mu);B_{2\rho}\times B_{2\varepsilon})=\deg(f_{h}^{\prime}(\mu_{0}%
-\rho)(x-x_{0});B_{2\varepsilon})-\deg(f_{h}^{\prime}(\mu_{0}+\rho
)(x-x_{0});B_{2\varepsilon})\text{.}%
\]
Therefore, the first part of the proof follows from the fact that
\[
\deg(f_{h}^{\prime}(\mu)(x-x_{0});B_{2\varepsilon})=n_{h}(\mu).
\]

Now, supposing $\eta_{h}(\mu_{0})\neq0$, for small $\varepsilon$ there is a
$(x_{\varepsilon},\mu_{\varepsilon})$ with $x_{\varepsilon}\subset
W^{\tilde{D}_{h}}$ such that $f_{h}(x_{\varepsilon},\mu_{\varepsilon})=0$ and
$d(x_{\varepsilon},x_{0})=\varepsilon.$ Moreover, when we let $\varepsilon$
tend to zero, by the compactness we have a series $\varepsilon_{k}%
\rightarrow0$ such that $\mu_{\varepsilon_{k}}\rightarrow\mu_{1}$. By the
continuity we conclude that $\mu_{1}=\mu_{0}$.
\end{proof}

When only one of the blocks $B_{k}$ has a determinant which changes sign, we
have the following result.

\begin{theorem}
\label{EnLoBi}For $k\in\{1,...,[n/2],n\}$, let $h$ be the maximum common
divisor of $k$ and $n$. Supposing $\sigma_{k}(\mu)$ changes sign at $\mu_{0}$
and $\sigma_{j}(\mu_{0})\neq0$ for the others $j\in\lbrack1,n/2]\cap
(\mathbb{N}h)$, then there is a bifurcation with maximal symmetry $\tilde
{D}_{h}$. This means that the local bifurcation is in $W^{\tilde{D}_{h}%
}\backslash\cup_{\tilde{D}_{h}\subset H}W^{H}.$
\end{theorem}

\begin{proof}
To assure all the symmetries of the bifurcation, we apply the previous theorem
with $h$ the maximum common divisor of $k$ and $n$. By hypothesis the product%
\[
\sigma_{n}\prod_{j\in\lbrack1,n/2]\cap\mathbb{N}h~(j\neq k)}\sigma_{j}(\mu
_{0})
\]
is not zero, then $\eta_{h}(\mu_{0})=\pm2$. Henceforth, there is a bifurcation
in the fixed point space of $\tilde{D}_{h}$.

It only remains to prove that $\tilde{D}_{h}$ is the maximum group of
symmetries. Let $\tilde{D}_{p}$ be a group such that $\tilde{D}_{h}%
\subset\tilde{D}_{p}$, this means that $h$ divides $p$. Since $p$ does not
divide $k $, then%
\[
sgn~\det f_{p}^{\prime}(x_{0})=\sigma_{n}\prod_{j\in\lbrack1,n/2]\cap
\mathbb{N}p}\sigma_{j}\neq0\text{.}%
\]
Consequently, the linear map $f_{p}^{\prime}(x_{0})$ is invertible, and by the
implicit function theorem, we deduce the non existence of solutions in
$W^{\tilde{D}_{p}}$ near $(x_{0},\mu_{0})$.
\end{proof}

\subsection{Global bifurcation}

Now, we wish to prove a global bifurcation result, which is just an adaptation
of the Rabinowitz' alternative. Notice that this approach may not give all the
best information available for a global result, as an application of the
$\Gamma$-equivariant degree. But this equivariant degree (for this larger
group) presents strong technical difficulties.

Let us define $T$ as the set $\{(\Gamma x_{0},\mu):\mu\in\mathbb{R}\}$ with
$x_{0}=P^{-1}\bar{a}$. Let $S$ be the set of zeros of $f_{h}(x,\mu)$, we
define $G=S\backslash T$ as the nontrivial solution set. $\bar{G}\backslash
G\subset T$ is the bifurcation set and an element of $(x_{0},\mu_{0})\in
\bar{G}\backslash G$ is said to be a bifurcation point. Let $C\subset\bar{G}
$\ be the connected component of the bifurcation point $(x_{0},\mu_{0})$. Then
$C\cap T$ consists of the bifurcation points of the branch $C$.

We define the collision set as
\[
\Psi=\{x\in\mathbb{R}^{2(n+1)}:x_{i}=x_{j}~(i\neq j)\}\text{.}%
\]
Also let $\Lambda_{\rho}=\{\left\Vert \mu\right\Vert <\rho\}$ and
$\Omega_{\rho}$ be
\[
\Omega_{\rho}=\{x\in\mathbb{R}^{2(n+1)}:\left\Vert x\right\Vert <\rho
,\rho^{-1}<d(x,\Psi)\}.
\]
Since, for $\rho$ big enough, the set $\Omega_{\rho}$ is a big ball without a
small neighborhood of hyperplanes of codimension $2$, then the set
$\Omega_{\rho}$ is connected.

We say that the component $C$\ is admissible whenever it is contained in some
set $\Omega_{\rho}\times\Lambda_{\rho}$. Otherwise we say that $C$\ is
inadmissible and this corresponds to the cases where (a): the parameter $\mu$,
on the component, goes to infinity, or (b): the norm of $x$ on the component
goes to infinity or (c): the component ends at a collision point.

\begin{theorem}
\label{EnGlBi} If the component $C$\ is admissible and the set $C\cap T$ is
isolated, then $C$\ returns to other bifurcation points $\{(x_{1},\mu
_{1}),...,(x_{r},\mu_{r})\}$ and%
\begin{equation}
\eta_{h}(x_{0},\mu_{0})+...+\eta_{h}(x_{r},\mu_{r})=0\text{.}%
\end{equation}

\end{theorem}

\begin{proof}
Since $C$ is admissible, we may construct a set $\bar{\Omega}\subset
\Omega_{\rho}\times\Lambda_{\rho}$ such that $C\subset\Omega$ with
$\partial\bar{\Omega}\cap C=\phi$ and such that $f_{h}(x,\mu)$ is zero on
$\partial\bar{\Omega}$ only when $x\in T$. Since $f_{h}(x,\mu)$ is not zero on
$\partial\Omega$ unless $x\in T$, then the degree $\deg(d(x,T)-\varepsilon
,f_{h};\Omega)$ is well defined. Moreover, as $\Omega$ is bounded, we can take
$\varepsilon$ big enough in such a way that this degree is zero.

By hypothesis $C\cap T$ consists of isolated points. Consequently, taking
$\varepsilon$ small enough, the points which satisfy $d(x,T)=\varepsilon$ and
$f_{h}(x,\mu)=0$ in $\Omega$ are in the finite and disjoint union of
$B_{2\varepsilon}(x_{0})\times B_{2\rho}(\mu_{0})$ for $(x_{0},\mu_{0})\in
C\cap T$. Hence, by the excision property of the degree we have%
\[
0=\deg(d(x,T)-\varepsilon,f_{h};\Omega)=\sum_{(x_{0},\mu_{0})\in C\cap T}%
\deg(d(x,x_{0})-\varepsilon,f_{h};B_{2\varepsilon}\times B_{2\rho})\text{.}%
\]

We conclude the result from the computation of the local degree.
\end{proof}

\subsection{Symmetries}

The relative equilibria with symmetry $\tilde{D}_{h}$ are composed of bodies
arranged as $n/h$ regular polygons of $h$ sides with some polygons related by
reflection. To give a sharper description, let us call an $h$-gon as the set of
positions%
\begin{equation}
\{re^{i\varphi}e^{k(2\pi i/h)}:k=1,...,h\}, \label{EqPo1}%
\end{equation}
and a $2h$-gon as%
\begin{equation}
\{re^{\pm i\varphi}e^{k(2\pi i/h)}z:k=1,...,h\}\text{.} \label{EqPo2}%
\end{equation}

\begin{proposition}
In a relative equilibrium with symmetries $\tilde{D}_{h}$ the central body
stays on the real axis if $h=1$ and remains at the origin if $h>1$. The other
bodies satisfy the following arrangements:

\begin{description}
\item[(a)] If $n/h$ is odd, the relative equilibrium has an $h$-gon
(\ref{EqPo1}) of bodies with $\varphi=0$. The remaining bodies form $2h$-gons
(\ref{EqPo2}), with $\varphi\in(0,\pi/h)$.

\item[(b)] If $n/h$ is even, the relative equilibrium has two $h$-gons
(\ref{EqPo1}) of bodies, one with $\varphi=0$ and another with $\varphi=\pi
/h$. The remaining bodies form $2h$-gons (\ref{EqPo2}), with $\varphi\in
(0,\pi/h)$.
\end{description}
\end{proposition}

\begin{proof}
Since the central body satisfies the symmetry $x_{0}=(\kappa,\kappa)x_{0}%
=\bar{x}_{0}$, then $x_{0}\in\mathbb{R}$. Moreover, as $x_{0}=(n/h)(\zeta
,\zeta)x_{0}=e^{-i2\pi/h}x_{0}$, then $x_{0}=0$ whenever $h>1$.

Now, for the remaining bodies, $j\in1,...,n$, we use the notation
$x_{j}=x_{j+kn}$. Hence these bodies satisfy the relations%
\begin{equation}
\text{(a) }x_{j}=(n/h)(\zeta,\zeta)x_{j}=e^{-i(2\pi/h)}x_{j+n/h}\text{ and (b)
}x_{j}=(\kappa,\kappa)x_{j}=\bar{x}_{-j}. \label{EqSy2}%
\end{equation}
By (a), the positions of the bodies are determined only by the bodies with
$j\in\mathbb{N\cap}(-n/2h,n/2h]$, and by (b), these are determined by the
bodies $j\in\mathbb{N\cap}[0,(n/2h)]$.

From (a) we have an $h$-gon (\ref{EqPo1}), for each $j\in\{0,(n/2h)\}$.
Actually, from (b), we deduce that $\varphi=0$ for $j=0$ and $\varphi=\pi/h $
for $j=n/2h$. Now, for each body, $j\in\mathbb{N\cap}(0,(n/2h))$, we have a
$2h$-gon (\ref{EqPo2}). Furthermore, since a $2h$-gon has collisions for
$\varphi\in\{0,\pi/h\}$, then we can chose $\varphi\in(0,\pi/h)$.
\end{proof}

In order to give examples of the previous descriptions, we shall analyze the
cases $\tilde{D}_{1}$, $\tilde{D}_{2}$ for $n$ even, and $\tilde{D}_{3}$ for
$n=6$.

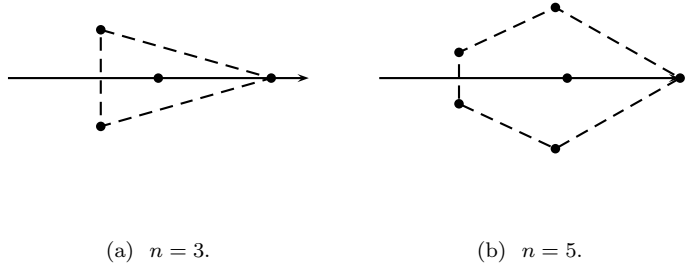
\begin{figure}[h]
\centering
\subfloat[ $n=3$.] { \begin{pspicture}(-2,-2)(2,2)
\SpecialCoor
\psline{->}(-2,0)(2,0)
\pspolygon[showpoints=true,linestyle=dashed](1.5,0)(1;140)(1;220)
\psdots(0,0)
\NormalCoor\end{pspicture}
} \qquad\subfloat[ $n=5$.] { \begin{pspicture}(-2,-2)(2,2)
\SpecialCoor
\psline{->}(-2,0)(2,0)
\pspolygon[showpoints=true,linestyle=dashed](2,0)(1;70)(1;160)(1;-160)(1;-70)
\psdots(.5,0)
\NormalCoor\end{pspicture}
}\caption{ Symmetries of the group $\tilde{D}_{1}$.}%
\end{figure}The group $\tilde{D}_{1}$ is a subgroup of $\tilde{D}_{n}$, and it
is generated by $(\kappa,\kappa)$. The relative equilibria with symmetry
$\tilde{D}_{1}$ have the central body on the real axis. The other bodies
satisfy the following:

\begin{description}
\item[(a)] If $n$ is odd. One body is on the real axis. The remaining bodies
form symmetric couples with respect to the real axis, (see the examples $n=3$
and $n=5$).

\item[(b)] If $n$ is even. Two bodies are on the real axis without any
relation. The remaining bodies form symmetric couples with respect to the real
axis, (see the examples $n=4$ and $n=6$).
\end{description}

\begin{figure}[h]
\centering
\subfloat[Symmetries of $\tilde{D}_{1}$.]{ 
\par
\begin{pspicture}(-2,-2)(2,2)\SpecialCoor
\psline{->}(-2,0)(2,0)
\pspolygon[showpoints=true,linestyle=dashed](1.5,0)(1;140)(-1,0)(1;220)
\psdots(1,0)
\NormalCoor\end{pspicture}}
\qquad\subfloat[Symmetries of $\tilde{D}_{2}$.] {
\begin{pspicture}(-2,-2)(2,2)\SpecialCoor
\psline{->}(-2,0)(2,0)
\pspolygon[showpoints=true,linestyle=dashed](1.5,0)(0,.7)(-1.5,0)(0,-.7)
\psdots(0,0)
\NormalCoor\end{pspicture}
}\caption{ $n=4$.}%
\end{figure}The group $\tilde{D}_{2}$ is generated by $(\pi,\pi)$ and
$(\kappa,\kappa)$, and it is a subgroup of $\tilde{D}_{n}$ whenever $n$ is
even. A relative equilibrium with symmetry $\tilde{D}_{2}$ has the central
body standing still at the origin. The other bodies satisfy the following:

\begin{description}
\item[(a)] If $n/2$ is odd. One pair of bodies is on the real axis symmetric
with respect to the imaginary axis . The remaining bodies form squares
symmetric with respect to both axes,(see the example $n=6$).

\item[(b)] If $n/2$ is even. One pair of bodies is on the real axis symmetric
with respect to the imaginary axis. Another pair of bodies is on the imaginary
axis symmetric with respect to the real axis. The remaining bodies form
squares which are symmetric with respect to both axes,(see the example $n=4$).
\end{description}

\begin{figure}[h]
\centering
\subfloat[
Symmetries of $\tilde{D}_{1}$.] {
\begin{pspicture}(-2,-2)(2,2)\SpecialCoor
\psline{->}(-2,0)(2,0)
\pspolygon[showpoints=true,linestyle=dashed](1,0)(1.5;70)(1;140)(-2,0)(1;-140)(1.5;-70)
\psdots(.5,0)
\NormalCoor\end{pspicture}
} \subfloat[
Symmetries of $\tilde{D}_{2}$.] {
\begin{pspicture}(-2,-2)(2,2)\SpecialCoor
\psline{->}(-2,0)(2,0)
\pspolygon[showpoints=true,linestyle=dashed](2,0)(1.3;30)(1.3;150)(-2,0)(1.3;-150)(1.3;-30)
\psdots(0,0)
\NormalCoor\end{pspicture}
} \subfloat[
Symmetries of $\tilde{D}_{3}$.] {
\begin{pspicture}(-2,-2)(2,2)\SpecialCoor
\psline{->}(-2,0)(2,0)
\pspolygon[showpoints=true,linestyle=dashed](2;0)(.5;60)(2;120)(.5;180)(2;240)(.5;300)
\psdots(0,0)
\NormalCoor\end{pspicture}
}\caption{ $n=6$.}%
\end{figure}

Finally, for the subgroup $\tilde{D}_{3}$ of $\tilde{D}_{6}$ we have $n/h=2$.
Therefore, the central body stands still at the origin and the remaining
bodies form two triangles without relation, one with $\varphi=0$ and the other
with $\varphi=$ $\pi/3$.

Note that these figures are only an illustration of the possible configurations
which may happen. They have to be taken in this perspective, as it the case of the
figures in other papers, such as \cite{MeSc88}. A precise numerical analysis of the
positions of the bodies far from the relative equilibria is outside our present
concern. Furthermore, our last proposition
is mathematically valid for any $\mu$. However, for the gravitational problem,
the masses need to be positive or, at least if one is considering an attraction
given by charges, that $\omega$ should be positive, since we took its square root.
Finally, the spectral analysis and some of the following remarks will give a
complement of information and a better justification of our figures.

\section{Spectral analysis}

It is time to calculate explicitly the bifurcation points for the general
potential (\ref{Eq13}). We begin by computing the matrices $A_{ij}$.

\begin{proposition}
Define $\alpha_{+}=(\alpha+1)/2$ and $\alpha_{-}=(\alpha-1)/2$, then, for
$n\geq3$, we have%
\begin{align*}
A_{00} & =\mu\left( s_{1}+\mu+\alpha_{-}n\right) I,\\
A_{n0} & =-\mu(\alpha_{-}I+\alpha_{+}R)\text{ and}\\
A_{nn} & =(s_{1}+\mu)I-\sum_{j=0}^{n-1}A_{nj}.
\end{align*}
In addition, we have for $j\in\{1,...,n-1\}$ that%
\[
A_{nj}=\frac{1}{\left( 2\sin(j\zeta/2)\right) ^{\alpha+1}}\left(
-\alpha_{-}I+\alpha_{+}e^{jJ\zeta}R\right) .
\]

\end{proposition}

\begin{proof}
Notice that $\nabla_{x_{i}}\phi(\left\Vert x_{i}-x_{j}\right\Vert
)=-\nabla_{x_{j}}\phi(\left\Vert x_{i}-x_{j}\right\Vert )$ for $i\neq j$,
thus,
\[
A_{ij}=\mu_{i}\mu_{j}D_{x_{j}}\nabla_{x_{i}}\phi(\left\Vert a_{i}%
-a_{j}\right\Vert )=-\mu_{i}\mu_{j}D_{x_{i}}^{2}\phi(\left\Vert a_{i}%
-a_{j}\right\Vert )\text{.}%
\]
And, for the matrix $A_{ii}$, we have%
\[
A_{ii}=(s_{1}+\mu)\mu_{i}I+\sum_{j\neq i}\mu_{i}\mu_{j}D_{x_{i}}^{2}%
\phi(\left\Vert a_{i}-a_{j}\right\Vert )=(s_{1}+\mu)\mu_{i}I-\sum_{j\neq
i}A_{ij}\text{.}%
\]

Let us set $a_{j}=(x_{j},y_{j})$ and $d_{ij}=\left\Vert (x_{i},y_{i}%
)-(x_{j},y_{j})\right\Vert $, then the function $\phi_{\alpha}(d_{ij})$ has
its matrix of second derivatives%
\[
D^{2}\phi_{\alpha}(d_{ij})=\frac{\alpha+1}{d_{ij}^{\alpha+3}}\left(
\begin{array}
[c]{cc}%
(x_{i}-x_{j})^{2} & (x_{i}-x_{j})(y_{i}-y_{j})\\
(x_{i}-x_{j})(y_{i}-y_{j}) & (y_{i}-y_{j})^{2}%
\end{array}
\right) -\frac{1}{d_{ij}^{\alpha+1}}I\text{.}%
\]

Since the distance from $a_{0}=(0,0)$ to $a_{n}=(1,0)$ is $d_{n0}=1$, then%
\[
A_{n0}=-\mu\left(
\begin{array}
[c]{cc}%
\alpha & 0\\
0 & -1
\end{array}
\right) =-\mu(\alpha_{-}I+\alpha_{+}R)\text{.}%
\]
Moreover, as $\sum_{j=1}^{n}e^{2jJ\zeta}=0$ for $\zeta\neq\pi$, or
equivalently $n\geq3$, then%
\[
-\sum_{j=1}^{n}A_{0j}=-\sum_{j=1}^{n}e^{jJ\zeta}A_{n0}e^{-jJ\zeta}=\mu
n\alpha_{-}I\text{.}%
\]
Therefore%
\[
A_{00}=(s_{1}+\mu)\mu I-\sum_{j=1}^{n}A_{0j}=\mu\left( s_{1}+\mu+\alpha
_{-}n\right) I.
\]

It remains only to find the matrix $A_{nj}$ for $j\in\{1,...,n-1\}$. As
$a_{n}=(1,0)$ and $a_{j}=(\cos j\zeta,\sin j\zeta)$, then the distance $d_{nj}
$ satisfies%
\[
d_{nj}^{2}=(1-\cos j\zeta)^{2}+\sin^{2}j\zeta=4\sin^{2}(j\zeta/2)\text{.}%
\]
Using the previous results, we have%
\[
A_{nj}=-\frac{\alpha+1}{d_{nj}^{\alpha+3}}\left(
\begin{array}
[c]{cc}%
(1-\cos j\zeta)^{2} & -(1-\cos j\zeta)\sin j\zeta\\
-(1-\cos j\zeta)\sin j\zeta & (\sin j\zeta)^{2}%
\end{array}
\right) +\frac{1}{d_{nj}^{\alpha+1}}I.
\]
Now, since $\sin^{2}j\zeta=(1-\cos j\zeta)(1+\cos j\zeta)$ and $d_{nj}%
^{2}=2(1-\cos j\zeta)$, then%
\[
A_{nj}=\frac{1}{d_{nj}^{\alpha+1}}\left( I-\frac{\alpha+1}{2}\left(
\begin{array}
[c]{cc}%
1-\cos j\zeta & -\sin j\zeta\\
-\sin j\zeta & 1+\cos j\zeta
\end{array}
\right) \right) \text{.}%
\]
Finally, using $d_{nj}=2\sin(j\zeta/2)$ we conclude the result.
\end{proof}

It can be seen that the bifurcation points are just the points where $\det
B_{k}$ changes sign. Now we can find explicitly the blocks $B_{k}$ for the
general potential (\ref{Eq13}).

\begin{proposition}
\label{EnBk2} Define $s_{k}$, $\alpha_{k}$, $\beta_{k}$\ and $\gamma_{k}$ as%
\[
s_{k}=\frac{1}{2^{\alpha}}\sum_{j=1}^{n-1}\frac{\sin^{2}(kj\zeta/2)}%
{\sin^{\alpha+1}(j\zeta/2)}\text{,}%
\]%
\[
\alpha_{k}=\frac{\alpha_{-}}{2}(s_{k+1}+s_{k-1})\text{, }\beta_{k}=\alpha
_{+}(s_{k}-s_{1})\text{ and }\gamma_{k}=\frac{\alpha_{-}}{2}(s_{k+1}%
-s_{k-1})\text{.}%
\]
For $k\in\{2,...,n-2,n\}$, the blocks $B_{k}$ are%
\[
B_{k}=\alpha_{+}\mu(I+R)+(s_{1}+\alpha_{k})I-\beta_{k}R-\gamma_{k}iJ\text{.}%
\]

\end{proposition}

\begin{proof}
From the definition of $B_{k}$ and the computation of $A_{nn}$, we have
\[
B_{k}=(s_{1}+\mu)I-A_{n0}+\sum_{j=1}^{n-1}A_{nj}(e^{j(ikI+J)\zeta}-I)\text{.}%
\]
And, from the computation of $A_{n0}$, we obtain that $B_{k}=\alpha_{+}%
\mu(I+R)+s_{1}I+D_{k}$, with%
\[
D_{k}=\sum_{j=1}^{n-1}A_{nj}(e^{j(ikI+J)\zeta}-I)\text{.}%
\]

Now, our problem has been reduced to calculate $D_{k}$. Using the explicit
computation of $A_{nj}$, we see that $D_{k}$ satisfies
\[
D_{k}=\sum_{j=1}^{n-1}\frac{\left( -\alpha_{-}I+\alpha_{+}e^{jJ\zeta
}R\right) (e^{j(ikI+J)\zeta}-I)}{\left( 2\sin(j\zeta/2)\right) ^{\alpha+1}%
}\text{.}%
\]
The coefficient of the sum can be written as%
\[
\alpha_{-}(I-e^{j(ikI+J)\zeta})-\alpha_{+}R\left( e^{-Jj\zeta}-e^{ijk\zeta
}\right) \text{.}%
\]
Notice that, using the equalities%
\begin{align*}
e^{-(jJ\zeta)}+e^{(jJ\zeta)} & =2I\cos j\zeta\text{ and}\\
e^{j(ikI+J)\zeta}+e^{-j(ikI+J)\zeta} & =2\left[ I\cos jk\zeta\cos
j\zeta+iJ\sin jk\zeta\sin j\zeta\right] ,
\end{align*}
we may cancel terms from the sum $D_{k}$ for $j$ and $n-j$ . In this way, we
obtain that the matrix $D_{k}$ is%
\[
\sum_{j=1}^{n-1}\frac{\alpha_{-}\left( I[1-\cos kj\zeta\cos j\zeta]-iJ[\sin
jk\zeta\sin j\zeta]\right) -\alpha_{+}R[\cos j\zeta-\cos jk\zeta]}{\left(
2\sin(j\zeta/2)\right) ^{\alpha+1}}\text{.}%
\]

Hence, we may write $D_{k}$ as $D_{k}=\alpha_{k}I-\beta_{k}R-\gamma_{k}iJ$
with%
\begin{align*}
\alpha_{k} & =\alpha_{-}\sum\frac{1-\cos kj\zeta\cos j\zeta}{\left(
2\sin(j\zeta/2)\right) ^{\alpha+1}}\text{,}\\
\beta_{k} & =\alpha_{+}\sum\frac{\cos j\zeta-\cos jk\zeta}{\left(
2\sin(j\zeta/2)\right) ^{\alpha+1}}\text{,}\\
\gamma_{k} & =\alpha_{-}\sum\frac{\sin jk\zeta\sin j\zeta}{\left(
2\sin(j\zeta/2)\right) ^{\alpha+1}}\text{.}%
\end{align*}
Finally, we conclude that $\alpha_{k}$, $\beta_{k}$ and $\gamma_{k}$ coincide
with the definitions in the proposition from the equalities%
\begin{align*}
\alpha_{k}+\gamma_{k} & =\alpha_{-}\sum\frac{1-\cos((k+1)j\zeta)}{\left(
2\sin(j\zeta/2)\right) ^{\alpha+1}}=\alpha_{-}s_{k+1}\text{,}\\
\alpha_{k}-\gamma_{k} & =\alpha_{-}\sum\frac{1-\cos((k-1)j\zeta)}{\left(
2\sin(j\zeta/2)\right) ^{\alpha+1}}=\alpha_{-}s_{k-1}\text{,}\\
\beta_{k} & =\alpha_{+}\sum2\frac{\sin^{2}(jk\zeta/2)-\sin^{2}(j\zeta
/2)}{\left( 2\sin(j\zeta/2)\right) ^{\alpha+1}}=\alpha_{+}(s_{k}%
-s_{1})\text{.}%
\end{align*}

\end{proof}

\begin{proposition}
For $k\in\{1,n-1\}$, we have that $B_{n-1}=\bar{B}_{1}$ and%
\[
B_{1}=\left(
\begin{array}
[c]{ccc}%
\mu\left( s_{1}+\mu+n\alpha_{-}\right) & -\left( \frac{n}{2}\right)
^{1/2}\mu\alpha & -\left( \frac{n}{2}\right) ^{1/2}\mu i\\
-\left( \frac{n}{2}\right) ^{1/2}\mu\alpha & s_{1}+\alpha_{1}+(\alpha+1)\mu
& \alpha_{1}i\\
\left( \frac{n}{2}\right) ^{1/2}\mu i & -\alpha_{1}i & s_{1}+\alpha_{1}%
\end{array}
\right) \text{.}%
\]

\end{proposition}

\begin{proof}
From the proof of the previous proposition, we have that%
\[
\sum_{j=1}^{n}A_{nj}e^{j(iI+J)\zeta}=\alpha_{+}(I+R)\mu+(s_{1}+\alpha
_{1})I-\beta_{1}R-\gamma_{1}iJ\text{.}%
\]
And, since $\beta_{1}=0$ and $\alpha_{1}=\gamma_{1}$, then%
\[
\sum_{j=1}^{n}A_{nj}e^{j(iI+J)\zeta}=\left(
\begin{array}
[c]{cc}%
s_{1}+\alpha_{1}+2\alpha_{+}\mu & \alpha_{1}i\\
-\alpha_{1}i & s_{1}+\alpha_{1}%
\end{array}
\right) .
\]
Moreover, since $Rv_{1}=\bar{v}_{1}$, then%
\[
n^{1/2}A_{0n}v_{1}=-n^{1/2}\mu(\alpha_{-}v_{1}+\alpha_{+}\bar{v}_{1}%
)=\mu\left( \frac{n}{2}\right) ^{1/2}\left(
\begin{array}
[c]{c}%
-\alpha\\
i
\end{array}
\right) .
\]
From the definition of $B_{1}$ we get the result. Finally, using the
computation of $B_{1}$, we may prove that $RB_{1}R=\bar{B}_{1}$, and then that
$B_{n-1}=RB_{1}R=\bar{B}_{1}$.
\end{proof}

Clearly, the sums $s_{k}$ are positive and satisfy $s_{k}=s_{n+k}=s_{n-k}$. To
analyze the bifurcation points, we need the following recursive formula for
$s_{k}$.

\begin{lemma}
Let $\bar{s}_{k}$ be defined as the sum $s_{k}$ but with $\alpha-2$ instead of
$\alpha$. Then, the sums $s_{k}$\ satisfy the recurrence formulae%
\[
s_{k+1}-s_{k}=(2k+1)s_{1}-\sum_{h=1}^{k}\bar{s}_{h}%
\]

\end{lemma}

\begin{proof}
We write the sum $s_{k}$ as%
\[
2^{\alpha}s_{k}=\sum_{j=1}^{n-1}\frac{1}{\sin^{\alpha-1}(j\zeta/2)}%
\frac{1-\cos(kj\zeta)}{1-\cos(j\zeta)}\text{.}%
\]
Using geometric series, we have%
\[
\frac{1-\cos(kj\zeta)}{1-\cos(j\zeta)}=\frac{1-e^{ijk\zeta}}{1-e^{ij\zeta}%
}\frac{1-e^{-ijk\zeta}}{1-e^{-ij\zeta}}=\sum_{l=0}^{k-1}\sum_{m=0}%
^{k-1}e^{ij(l-m)\zeta}.
\]
Now, we may cancel common terms from $s_{k+1}$ and $s_{k}$ as%
\[
2^{\alpha}(s_{k+1}-s_{k})=\sum_{j=1}^{n-1}\frac{1}{\sin^{\alpha-1}(j\zeta
/2)}\sum_{h=-k}^{k}e^{ijh\zeta}\text{.}%
\]
Finally, since%
\[
\sum_{h=-k}^{k}e^{ijh\zeta}=\sum_{h=-k}^{k}\cos jh\zeta=(2k+1)-4\sum_{h=1}%
^{k}\sin^{2}(jh\zeta/2)\text{,}%
\]
then%
\[
s_{k+1}-s_{k}=(2k+1)s_{1}-\sum_{h=1}^{k}\sum_{j=1}^{n-1}\frac{\sin^{2}%
(hj\zeta/2)}{2^{\alpha-2}\sin^{\alpha-1}(j\zeta/2)}=(2k+1)s_{1}-\sum_{h=1}%
^{k}\bar{s}_{h}\text{.}%
\]

\end{proof}

The idea of using geometric series is taken from \cite{CaSc00}, where it is
used to calculate $s_{k}$ for the vortex case $\alpha=1$. Iterating this
result we obtain the equalities%
\begin{equation}
s_{k+1}-2s_{k}+s_{k-1}=2s_{1}-\bar{s}_{k} \label{Eq412}%
\end{equation}
and%
\begin{equation}
s_{k}=\sum_{l=0}^{k-1}(s_{l+1}-s_{l})=\sum_{l=0}^{k-1}\left( (2l+1)s_{1}%
-\sum_{h=1}^{l}\bar{s}_{h}\right) =k^{2}s_{1}-\sum_{l=1}^{k-1}l\bar{s}_{k-l}.
\label{Eq411}%
\end{equation}

\subsection{General potential}

Now, we only need to find the bifurcation points, that is, the points where
$\sigma_{k}(\mu)$ changes sign for the general potential (\ref{Eq13}). For
$k=n$, we have that $\beta_{n}=-\alpha_{+}s_{1}$ and $\alpha_{n}=\alpha
_{-}s_{1}$, then $e_{1}^{T}B_{n}e_{1}=(\alpha+1)(\mu+s_{1})$ and
\[
\sigma_{n}(\mu)=sgn(\mu+s_{1})
\]

\begin{proposition}
The sign $\sigma_{1}$ is%
\[
\sigma_{1}(\mu)=sgn(b_{1}\mu(\mu+s_{1})(\mu-\mu_{1}))\text{,}%
\]
where $\mu_{1}=-a_{1}/b_{1}$, with%
\[
a_{1}=\left( s_{1}+2\alpha_{1}\right) \left( 2s_{1}+n\alpha-n\right)
\text{ and }b_{1}=\left( \alpha+1\right) \left( 2s_{1}+2\alpha
_{1}-n\right) \text{.}%
\]

\end{proposition}

\begin{proof}
We get the result from the fact that $\det B_{1}$ can be factored as follows
\begin{align*}
\frac{2\det B_{1}}{\mu\left( \mu+s_{1}\right) } & =\mu\left(
\alpha+1\right) \left( 2s_{1}+2\alpha_{1}-n\right) +\left( s_{1}%
+2\alpha_{1}\right) \left( 2s_{1}+n\alpha-n\right) \\
& =b_{1}(\mu-\mu_{1})\text{.}%
\end{align*}

\end{proof}

\begin{remark}
Notice that $\sigma_{1}$ and $\sigma_{n}$ change sign at $-s_{1}$, then
$\eta_{1}(-s_{1})=0$ and $\eta_{n}(-s_{1})=\pm2$. Nevertheless, there are two
explicit bifurcations at $-s_{1}$ one with symmetry $\tilde{D}_{1}$ and
another one with $\tilde{D}_{n}$. Indeed, the bifurcation with symmetry
$\tilde{D}_{1}$ is made of the translations of $\bar{a}$, $(0+r,e^{i\zeta
}+r,...,e^{in\zeta}+r)$, with $\omega=0$. The bifurcation with symmetry
$\tilde{D}_{n}$ is made of the homotheties of $\bar{a}$, $(0,e^{i\zeta
}r,...,e^{in\zeta}r)$, with $\omega=0$.

In addition, since $\sigma_{1}$ changes sign at $\mu=0$, then $\eta_{1}%
(0)=\pm2$ . Therefore, there must be a bifurcation, with symmetry $\tilde
{D}_{1}$, at $\mu=0$. As the central body has mass zero, $\mu=0$, then the
bifurcation has no physical meaning since it is made of the solutions
$(r,e^{i\zeta},...,e^{in\zeta})$, with $\mu=0$.
\end{remark}

\begin{proposition}
For $k\in\{2,...,[n/2]\}$, the signs $\sigma_{k}$ are
\[
\sigma_{k}(\mu)=sgn~(\mu-\mu_{k}),
\]
where $\mu_{k}=-a_{k}/b_{k}$, with%
\[
a_{k}=(s_{1}+\alpha_{k})^{2}-\gamma_{k}^{2}-\beta_{k}^{2}\text{ and }%
b_{k}=\left( \alpha+1\right) (s_{1}+\alpha_{k}+\beta_{k})\text{.}%
\]

\end{proposition}

\begin{proof}
For $k=n/2$, we have that $\gamma_{n/2}=0$ and%
\[
e_{1}^{T}B_{n/2}e_{1}=s_{1}+\alpha_{n/2}-\beta_{n/2}+(\alpha+1)\mu.
\]
For $k\notin\{1,n/2,n\}$, the determinant of $B_{k}$ is%
\[
\det B_{k}=(\alpha_{+}\mu+(s_{1}+\alpha_{k}))^{2}-\gamma_{k}^{2}-(\alpha
_{+}\mu-\beta_{k})^{2}=b_{k}\mu+a_{k}\text{.}%
\]
From the definitions of $\alpha_{k}$ and $\beta_{k}$, we have that%
\[
s_{1}+\alpha_{k}+\beta_{k}=s_{k}+(\alpha_{-}/2)\left( s_{k+1}+2s_{k}%
+s_{k-1}-2s_{1}\right) \text{.}%
\]
Using the equality (\ref{Eq412}) and the fact that $4s_{k}-\bar{s}_{k}$ is
positive, we get the inequality $s_{k+1}+2s_{k}+s_{k-1}>2s_{1}$. Consequently,
the factor $b_{k}$ is positive and we may conclude the result.
\end{proof}

From the discussion in the previous remark, the true bifurcations are found at
$\mu_{k}$ for $k\in\{1,...,[n/2]\}$. From the bifurcation theorem we have the following:

\begin{theorem}
For $k\in\{1,...,[n/2]\}$, let $h$ be the maximum common divisor of $k$ and
$n$. If $\mu_{k}$ is different from $-s_{1}$, $0$ and $\mu_{j}$ for the other
$j\in\lbrack1,n/2]\cap h\mathbb{N}$, then, from $\mu_{k}$, there is a global
bifurcation of relative equilibria with maximal symmetry $\tilde{D}_{h}$.
\end{theorem}

By maximal symmetry $\tilde{D}_{h}$ we mean that the local branch has symmetry
$\tilde{D}_{h}$ but not for a bigger group $\tilde{D}_{p}$.

By global bifurcation we mean that, whenever the branch is admissible, the
branch returns to other bifurcation points and the sum of the local degrees at
these bifurcation points is zero. The branch is inadmissible when the
parameter or the norm goes to infinity, or when the branch ends in a collision solution.

\begin{remark}
Notice that these results are applicable only for $n\geq3$, since the
irreducible representations of the definition (\ref{EnVk}) are not consistent
for $n=2$. Nevertheless, the case $n=2$ was analyzed in the same spirit
in a previous remark. For instance, we have proved that there is a
bifurcation of relative equilibria with symmetry $\tilde{D}_{1}$ from $\mu
_{1}=-\left( 2\alpha+s_{1}\right) /(\alpha+1)$. Also we did calculate, for
the vortex problem, that $\mu_{1}=-5/4$ and, for the body problem, $\mu
_{1}=-17/12$.
\end{remark}

\subsection{$(n+1)$-vortex potential}

Here we give a short description of the bifurcation points for the vortex
problem, since, in this case, we can calculate explicitly the bifurcation
points $\mu_{k}$.

\begin{proposition}
For $\alpha=1$, we have that
\[
s_{k}=k(n-k)/2.
\]

\end{proposition}

\begin{proof}
For $\alpha=1$, we have $s_{1}=(n-1)/2$. In addition, we may calculate
$\bar{s}_{k}$ as%
\[
\bar{s}_{k}=2\sum_{j=1}^{n-1}\sin^{2}(kj\zeta/2)=\sum_{j=1}^{n-1}%
(1-\cos(kj\zeta))=n.
\]
Therefore, from the formula (\ref{Eq411}), we have that
\[
s_{k}=k^{2}\left( n-1\right) /2-n\sum_{l=1}^{k-1}l=k(n-k)/2.
\]

\end{proof}

From the definitions with $\alpha=1$, we have $\alpha_{-}=0$, $\alpha_{+}=1 $,
$\alpha_{k}=0$, $\gamma_{k}=0$ and $\beta_{k}=s_{k}-s_{1}$. Since $\mu
_{k}=s_{k}/2-s_{1}$, for $k\in\{2,...,[n/2]\}$, then%
\[
\mu_{k}=\left( -k^{2}+nk-2n+2\right) /4\text{.}%
\]
And, for $k=1$, we have%
\[
\mu_{1}=s_{1}^{2}=(n-1)^{2}/4.
\]

Consequently, the bifurcation point $\mu_{2}=-1/2$ is always negative and
$\mu_{3}=(n-7)/4$ is positive only for $n\geq8$. Moreover, since $\mu_{k}$ is
increasing in $n$ for $k\geq3$, then $\mu_{k}$ is always positive for $k\geq4
$. Notice also that the bifurcation points $\mu_{k}$ are increasing in $k$ for
$k\in\{2,...,[n/2]\}$, and, as a consequence, the $\mu_{k}$ are different.

\begin{theorem}
For $n\geq3$, and each $k\in\{1,...,[n/2]\}$, the polygonal relative
equilibrium has a global bifurcation of relative equilibria from $\mu_{k}$
with maximal symmetry $\tilde{D}_{h}$.
\end{theorem}

The existence of the local bifurcation was proved before in the article
\cite{MeSc88}, with a normal form method.

\subsection{$(n+1)$-body potential}

Notice that, for the $(n+1)$-body problem, the equations have a physical
meaning only for $\mu\geq0$. Given that we cannot calculate explicitly the
sums $s_{k}$ in this case, we shall give an asymptotic computation of the sums
$s_{k}$ and of the bifurcation points $\mu_{k}$.

\begin{proposition}
For $n$ big enough, the bifurcation point $\mu_{1}$ is negative and $\mu_{k} $
is positive for $k\geq2$.
\end{proposition}

\begin{proof}
For the $(n+1)$-body problem $\alpha=2$. From the definitions, we have in this
case $\alpha_{-}=1/2$, $\alpha_{+}=3/2$, $\alpha_{k}=(s_{k+1}+s_{k-1})/4 $,
$\beta_{k}=3(s_{k}-s_{1})/2$ and $\gamma_{k}=(s_{k+1}-s_{k-1})/4$.

Using integral estimates, it can be easily seen that $s_{1}/n\rightarrow
\infty$ and that $\bar{s}_{k}/n$ is finite when $n$ goes to infinity.
Therefore, from the formula (\ref{Eq411}), we have the limits $s_{k}%
/s_{1}\rightarrow k^{2}$, when $n$ goes to infinity.

We have, for $k\geq2$, that $\ \beta_{k}/s_{1}\rightarrow3(k^{2}-1)/2$,
$\alpha_{k}/s_{1}\rightarrow(k^{2}+1)/2$ and $\gamma_{k}/s_{1}\rightarrow k $,
when $n\rightarrow\infty$. Therefore, from the definitions of $a_{k}$ and
$b_{k}$, we obtain the limits $b_{k}/s_{1}\rightarrow6k^{2}$ and%
\[
a_{k}/s_{1}^{2}=(1+\alpha_{k}/s_{1})^{2}-(\beta_{k}/s_{1})^{2}-(\gamma
_{k}/s_{1})^{2}\rightarrow-k^{2}\left( 2k^{2}-5\right) \text{.}%
\]
Consequently, the result follows from the fact that $\mu_{k}/s_{1}$ converges
to the positive limit $\left( 2k^{2}-5\right) $ for $k\geq2$.

For $k=1$, we have that $\alpha_{1}/s_{1}\rightarrow1$, then we obtain the
result from%
\[
\mu_{1}/s_{1}=-\frac{\left( s_{1}+2\alpha_{1}\right) \left( 2s_{1}%
+n\right) }{3\left( 2s_{1}+2\alpha_{1}-n\right) }\rightarrow-1/2.
\]

\end{proof}

In \cite{MeSc88}, the bifurcation of the local branch from $\mu_{k}$ is proven
for the $(n+1)$- body problem.

\begin{remark}
Given the numerical evidence of $\mu_{k}$, for instance see \cite{MeSc88}, it
seems that $\mu_{1}\geq0$ for $n\in\{3,4,5,6\}$, $\mu_{2}\geq0$ for $n\geq10$
and $\mu_{k}\geq0$ for every $k\geq3$. The numerical evidence also suggests
that the $\mu_{k}$ are increasing for $k\in\{2,...,[n/2]\}$. This is true at
least in the limit when $n\rightarrow\infty$, because $(\mu_{k+1}-\mu
_{k})/s_{1}$ converges to the positive limit $\left( 2k+1\right) /3$.
\end{remark}

\begin{theorem}
Assuming the numerical evidence of the previous remark, from $\mu_{1}$ for
$n\in\{3,4,5,6\}$, from $\mu_{2}$ for $n\geq10$, and from $\mu_{k}$ for each
$k\in\{3,...,[n/2]\}$, the polygonal relative equilibrium has a global
bifurcation of relative equilibria with maximal symmetry $\tilde{D}_{h}$.
\end{theorem}

\section{dNLS}

The dNLS equations are
\[
i\dot{q}_{j}=h(\left\Vert q_{j}\right\Vert ^{2})q_{j}+(q_{j+1}-2q_{j}%
+q_{j-1})\text{,}%
\]
where $q_{j}\in\mathbb{C}$ represents the oscillator and $h$ is the nonlinear
potential. We wish to study a finite circular lattice, that is, a lattice of
oscillators for $j\in\{1,...,n\}$, with periodic conditions $q_{j}=q_{j+n}$.

The solutions of the form $q_{j}=e^{\omega ti}x_{j}$, with $x_{j}$ constant,
are called relative equilibria. In order to obtain the amplitude as a
parameter, we need to change coordinates, with $q_{j}=\mu e^{\omega ti}x_{j}$.
In this manner, we have that the values $x_{j}$ form a relative equilibrium
when%
\[
-\omega x_{j}=h(\left\vert \mu x_{j}\right\vert ^{2})x_{j}+(x_{j+1}%
-2x_{j}+x_{j-1})\text{.}%
\]

\begin{remark}
Given that the lattice is integrable for $n=1$ and $n=2$, we shall look for
bifurcation of relative equilibria for $n\geq3$. Actually, according to
\cite{EiJo03}, it is possible to find all the bifurcation diagrams of the
relative equilibria for $n\leq4$. Notice that the relative equilibria are
known as breathers when they are localized.
\end{remark}

The starting point is a relative equilibrium which looks like a rotating wave
and is the equivalent of the polygonal relative equilibrium in the $n$-body
problem. We give next a condition which needs to be satisfied by the potential
for the existence of this rotating wave.

\begin{proposition}
Define $a_{j}=e^{ij\zeta}$, with $\zeta=2\pi/n$, then $\bar{a}=(a_{1}%
,...,a_{n})$ is a relative equilibrium if%
\[
\omega=4\sin^{2}(\zeta/2)-h(\mu^{2})\text{.}%
\]

\end{proposition}

\begin{proof}
Since $a_{j+1}-2a_{j}+a_{j-1}=-4\sin^{2}(\zeta/2)a_{j}$, then%
\[
V_{x_{j}}(\bar{a})=(\omega+h(\mu^{2})-4\sin^{2}(\zeta/2))a_{j}\text{.}%
\]

\end{proof}

\begin{remark}
Note that the existence of the rotating wave is determined by a
non-homogeneous relation between the amplitude $\mu$ and the frequency
$\omega$. This is different from the $n$-body problem, where the existence of
the relative equilibrium is determined by a homogeneous relation.
\end{remark}

In order to show the similarities with the $n$-body problem we change to real
coordinates. Let $x=(x_{1},...,x_{n})\in\mathbb{R}^{2n}$ be the vector of
positions, then the relative equilibria are critical points of the potential%
\[
V(x)=\frac{1}{2}\sum_{j=1}^{n}\left\{ H(x_{j},\mu)-\left\vert x_{j+1}%
-x_{j}\right\vert ^{2}\right\} ,
\]
where $x_{j}=x_{j+n}$ and $H(x,\mu)$ is a function such that $\nabla
H(x)=\omega x+h(\left\vert \mu x\right\vert ^{2})x$.

From the point of view of the symmetries, there is practically no difference
with the definitions of the $n$-body problem (\ref{EnAc}). The unique
difference is in the fact that we are not including the coordinate $x_{0}$ of
the $n$-body problem. So, in this case, the group $D_{n}$ acts on
$\mathbb{R}^{2n}$ as $\rho(\gamma)(x_{1},...,x_{n})=(x_{\gamma(1)}%
,...,x_{\gamma(n)})$ and the group $O(2)=S^{1}\cup\kappa S^{1}$ in a similar
way. The fact that the gradient $\nabla V$ is $D_{n}$-equivariant follows from
the periodicity conditions $x_{j}=x_{j+n}$. Moreover, it is well known that
the potential is invariant when we rotate the phases of all oscillators, so
the gradient $\nabla V$ is $O(2)$-equivariant.

As a consequence, we may adapt the results of sections two and three.
Actually, as in Proposition (\ref{EnBk}), in this case the blocks are given
by
\begin{equation}
B_{k}=\sum_{j=1}^{n}A_{nj}e^{j(ikI+J)\zeta} \label{Eq501}%
\end{equation}
for $k\in\{1,...,n/2,n\}$, and the signs $\sigma(\mu)$ are defined as before in
(\ref{EqSg}).\ Furthermore, since, in this case, there is no collision points,
then the bifurcation is inadmissible only when the parameter $\mu$ or the norm
of the branch goes to infinity.

We wish to describe briefly the meaning of the symmetries $\tilde{D}_{h}$
(\ref{EqSy2}) for the dNLS equations. Due to $x_{j}=e^{-i(2\pi/h)}x_{j+n/h}$,
then the solutions look like rotating waves composed of $h$ identical waves,
each one formed by $\bar{n}=n/h$ oscillators which satisfy the reflection
symmetry $x_{j}=\bar{x}_{\bar{n}-j}$. An example of relative equilibria with
symmetry $\tilde{D}_{h}$ is%
\[
x_{j}=(1+\varepsilon\sin^{2}j(\pi/\bar{n}))e^{i(2\pi/h)}.
\]

Given that most of the work is already done, we shall focus our attention on
finding the bifurcation points.

\subsection{General potential}

Again, the first step is to find the submatrices $A_{ij}$ of $D^{2}V(x)$ at
$a_{j}$.

\begin{proposition}
The submatrices $A_{nj}$ are $A_{nj}=I$ for $j\in\{1,n-1\}$, $A_{nj}=0$ for
$j\notin\{1,n-1,n\}$ and
\[
A_{nn}=\left( -2\cos\zeta\right) I+2\mu^{2}h^{\prime}(\mu^{2}%
)diag(1,0)\text{.}%
\]

\end{proposition}

\begin{proof}
As the coupling is linear and only between adjacent oscillators, then
$A_{nj}=I$ for $j\in\{1,n-1\}$, $A_{nj}=0$ for $j\notin\{1,n-1,n\}$ and
$A_{nn}=D^{2}H(a_{n})-2I$.

Let $x_{0}=(x,y)$, since $\nabla H(x_{0})=\omega x_{0}+h(\left\vert \mu
x_{0}\right\vert ^{2})x_{0}$, then
\[
D^{2}H(x_{0})=(\omega+h(\left\vert \mu x_{0}\right\vert ^{2}))I+2\mu
^{2}h^{\prime}(\left\vert \mu x_{0}\right\vert ^{2})\left(
\begin{array}
[c]{cc}%
x^{2} & xy\\
xy & y^{2}%
\end{array}
\right) .
\]
Since $\bar{a}$ is an equilibrium when $\omega+h(\mu^{2})=4\sin^{2}(\zeta/2)$,
then at $a_{n}=(1,0)$ we have%
\[
D^{2}H(a_{n})=4\sin^{2}(\zeta/2)I+2\mu^{2}h^{\prime}(\left\vert \mu\right\vert
^{2})diag(1,0)\text{.}%
\]
Hence, we conclude the result from the equality $4\sin^{2}(\zeta
/2)-2=-2\cos\zeta$.
\end{proof}

Now we may calculate the blocks $B_{k}$ from (\ref{Eq501}).

\begin{proposition}
Define $\alpha_{k}$\ and $\gamma_{k}$\ as%
\[
\alpha_{k}=4\cos\zeta\sin^{2}k\zeta/2\text{ and }\gamma_{k}=2\sin k\zeta
\sin\zeta\text{.}%
\]
Then, the blocks $B_{k}$ are%
\[
B_{k}=-\alpha_{k}I+\gamma_{k}(iJ)+2\mu^{2}h^{\prime}(\left\vert \mu\right\vert
^{2})diag(1,0).
\]

\end{proposition}

\begin{proof}
Using the explicit computation of $A_{nj}$, we have
\[
B_{k}=\left( e^{(ikI+J)\zeta}+e^{-(ikI+J)\zeta}\right) +\left( -2\cos
\zeta\right) I+2\mu^{2}h^{\prime}(\mu^{2})diag(1,0).
\]
Then, from the equalities $2\cos\zeta(\cos k\zeta-1)=-\alpha_{k}$ and%
\[
e^{(ikI+J)\zeta}+e^{-(ikI+J)\zeta}=(2\cos k\zeta\cos\zeta)I+(2\sin k\zeta
\sin\zeta)iJ\text{,}%
\]
we obtain the form of $B_{k}$.
\end{proof}

Now, it remains to find the bifurcation points. Since $\alpha_{n}=0$, then we
have $e_{1}^{T}B_{n}e_{1}=2\mu^{2}h^{\prime}(\mu^{2})$ and%
\[
\sigma_{n}=sgn~h^{\prime}(\mu^{2}).
\]

Since, for $n=3$, we have $\alpha_{1}=-\gamma_{1}<0$, then $\sigma
_{1}=sgn~h^{\prime}(\mu^{2})$. As for $n=4$, we have $\alpha_{k}=0$ and
$\gamma_{1}\neq0$, then $\sigma_{1}=-1$ and $\sigma_{2}=sgn~(h^{\prime}%
(\mu^{2}))$. Given that, in our examples, $h^{\prime}(\mu^{2})$ does not
change sign, there are no bifurcation points for $n=3,4$.

Consequently, we shall focus our attention only on the cases $n\geq5$, for
$k\in\{1,...,n/2\mathbb{\}}$, where we can assume
\[
\alpha_{k}=4\cos\zeta\sin^{2}k\zeta/2\geq0.
\]

\begin{proposition}
For $n\geq5$ and $k\in\{1,...,n/2\mathbb{\}}$, the sign $\sigma_{k}$ can
change sign only for the solutions of $\mu^{2}h^{\prime}(\mu^{2})=\delta_{k}$,
with%
\[
\delta_{k}=(\alpha_{k}^{2}-\gamma_{k}^{2})/(2\alpha_{k}).
\]
Moreover, we have $\delta_{1}<0$, $\delta_{2}=0$ and $\delta_{k}>0$ for
$k\in\lbrack3,n/2]\cap\mathbb{N}$.
\end{proposition}

\begin{proof}
For $k\in\lbrack1,n/2)\cap\mathbb{N}$, we have%
\begin{align*}
\det B_{k} & =\alpha_{k}^{2}-\gamma_{k}^{2}-2\alpha_{k}\mu^{2}h^{\prime}%
(\mu^{2})\\
& =2\alpha_{k}\left( \delta_{k}-\mu^{2}h^{\prime}(\mu^{2})\right) \text{.}%
\end{align*}
Since $\alpha_{k}>0$, then $\sigma_{k}=sgn~\left( \delta_{k}-\mu^{2}%
h^{\prime}(\mu^{2})\right) $.

When $k=n/2$, we have $\gamma_{n/2}=0$, then $e_{1}^{T}B_{n/2}e_{1}%
=-\alpha_{n/2}+2\mu^{2}h^{\prime}(\mu^{2})$. Thus, $\sigma_{n/2}=sgn~\left(
\mu^{2}h^{\prime}(\mu^{2})-\alpha_{n/2}\right) $ changes sign only for the
solutions of $\mu^{2}h^{\prime}(\mu^{2})=\alpha_{n/2}/2=\delta_{n/2}$.

Finally, since $\delta_{k}$ has the sign of%
\[
\alpha_{k}^{2}-\gamma_{k}^{2}=16(\sin^{2}k\zeta/2-\sin^{2}\zeta)\sin^{2}%
k\zeta/2\text{,}%
\]
then $\delta_{k}$ has the sign of $\sin^{2}k\zeta/2-\sin^{2}\zeta$.
\end{proof}

From the bifurcation theorems (\ref{EnLoBi}) and (\ref{EnGlBi}), we have the
following result:.

\begin{theorem}
For each simple solution of $\mu_{k}^{2}h^{\prime}(\mu_{k}^{2})=\delta_{k}$,
from the amplitude $\mu_{k}$ we have a global bifurcation of relative
equilibria with symmetry $\tilde{D}_{h}$, where $h$ is the maximum common
divisor of $k$ and $n$.
\end{theorem}

\begin{remark}
Actually, we may analyze more complex lattices whenever we preserve the
symmetries. For instance, we may consider nonlinear coupling and coupling with
distant oscillators.
\end{remark}

Now, we wish to give two typical examples.

\subsection{The Schr\"{o}dinger cubic potential}

For the cubic Schr\"{o}dinger potential, we need to set $h(x)=x$. In this case
$h^{\prime}(\mu^{2})=1$ and $\sigma_{n}=1$.

Then, for $n\geq5$, the sign $\sigma_{k}(\mu)$ changes only when $\mu
_{k}=\sqrt{\delta_{k}}$, if $\delta_{k}$ is positive. As we have proven before
that $\delta_{k}$ is positive when $n\geq6$, for $k\in\{3,...,[n/2]\mathbb{\}}%
$, and, since the numbers $\delta_{k}$ are increasing in $k$, then the
$\mu_{k}$ are increasing for $k\in\{3,...,[n/2]\mathbb{\}}$.

\begin{theorem}
For the cubic Schr\"{o}dinger potential, for $n\in\{6,7,..\}$, for each
$k\in\{3,...,[n/2]\}$ there is a global bifurcation of relative equilibria
with maximal symmetry $\tilde{D}_{h}$ from the amplitude $\sqrt{\delta_{k}}$.
\end{theorem}

\subsection{A saturable potential}

For a saturable potential, we need to set $h=(1+x)^{-1}$. In this case,
$h^{\prime}(\mu^{2})=-(1+\mu^{2})^{-2}$, $\sigma_{n}=-1$ and%
\[
\mu^{2}h^{\prime}(\mu^{2})=-\mu^{2}(1+\mu^{2})^{-2}%
\]
is a function with range $(-1/4,0)$ and a single minimum at $\mu^{2}=1$.
Therefore, there are two zeros, $\mu_{-}\in(0,1)$ and $\mu_{+}\in(1,\infty)$,
of the equation $\mu^{2}h^{\prime}(\mu^{2})=\delta_{k}$, when $\delta_{k}%
\in(-1/4,0)$.

Since we have proven before that $\delta_{k}\geq0$ for $k\geq2$, there is no
bifurcation for $k\geq2$ and it remains only to analyze the case $k=1$. Since
\[
\delta_{1}=2(\sin^{2}\zeta/2-\sin^{2}\zeta)/\cos\zeta\rightarrow0^{-}%
\]
when $n\rightarrow\infty$, then $\delta_{1}\in(-1/4,0)$ if $n$ is big enough.
Indeed, we obtain numerically that $\delta_{1}\in(-1/4,0)$ for $n\geq16$.

\begin{theorem}
For the lattice with saturable potential, for $n\geq16$, from the amplitudes
$\mu_{-}\in(0,1)$ and $\mu_{+}\in(1,\infty)$ there is a bifurcation of
relative equilibria with maximal symmetry $\tilde{D}_{1}$.
\end{theorem}

\begin{acknowledgement}

The authors wish to thank the referee for his comments.
Also, C.G-A wishes to thank the CONACyT for his scholarship
and J.I for the grant No. 133036.

\end{acknowledgement}





\bibliographystyle{model1-num-names}
\bibliography{bib}

\begin{thebibliography}{00}

\bibitem{CaSc00}
H.~E. Cabral and D.~S. Schmidt.
\newblock Stability of relative equilibria in the problem of $n+1$\ vortices.
\newblock {\em SIAM J. Math. Anal.}, 31(2):231--250, 2000.

\bibitem{ChOr02}
P.~Chossat, J.P. Ortega, and T.~S. Ratiu.
\newblock Hamiltonian {Hopf} bifurcation with symmetry.
\newblock {\em Arch. Ration. Mech. Anal..}, 163(1):1--33, 2002.

\bibitem{EiJo03}
J.C. Eilbeck and M.~Johansson.
\newblock The discrete nonlinear {Schr{\"o}dinger} equation -- 20 years on.
\newblock In Luis V{\'a}zquez, editor, {\em Proceedings of the 3rd conference
on localization and energy transfer in nonlinear systems}, pages 44--67. NJ:
World Scientific, Singapore, 2003.

\bibitem{Ga10}
C.~Garc\'{\i}a-Azpeitia.
\newblock Aplicaci\'{o}n del grado ortogonal a la bifurcaci\'{o}n en sistemas
hamiltonianos.
\newblock UNAM. PhD thesis, Mexico, 2010.

\bibitem{GaIz10}
C.~Garc\'{\i}a-Azpeitia and J.~Ize.
\newblock Global bifurcation of planar and spatial periodic solutions in the
restricted n-body problem.
\newblock To appear in Celestial Mechanics and Dynamical Astronomy, 2011.

\bibitem{Iz95}
J.~Ize.
\newblock Topological bifurcation.
\newblock In {\em Topological nonlinear analysis}, Progr. Nonlinear
Differential Equations Appl., 15, pages 341--463. Birkh{\"a}user, Boston,
1995

\bibitem{IzVi03}
J.~Ize and A.~Vignoli.
\newblock {\em Equivariant degree theory}.
\newblock De Gruyter Series in Nonlinear Analysis and Applications 8. Walter de
Gruyter, Berlin, New-York, 2003.

\bibitem{Jo04}
M.~Johansson.
\newblock Hamiltonian {Hopf} bifurcations in the discrete nonlinear
{Schr{\"o}dinger} trimer: oscillatory instabilities, quasiperiodic solutions
and a 'new' type of self-trapping transition.
\newblock {\em J. Phys. A: Math. Gen.}, 37:2201--2222, 2004.

\bibitem{MeHa91}
K.R Meyer and G.~R. Hall.
\newblock {\em An Introduction to Hamiltonian Dynamical Systems}.
\newblock Springer-Verlag, Berlin, 1991.

\bibitem{MeSc88}
K.~R. Meyer and D.~S. Schmidt.
\newblock Bifurcations of relative equilibria in the $n$-body and {Kirchhoff}
problems.
\newblock {\em SIAM J. Math. Anal.}, 19(6):1295--1313, 1988. 

\bibitem{Ne01}
P.~K. Newton.
\newblock {\em The $N$-vortex problem.}
\newblock Analytical techniques. Applied Mathematical Sciences, 145.
Springer-Verlag, New York, 2001.

\bibitem{PaDo09}
C.L. Pando and E.J Doedel.
\newblock Bifurcation structures and dominant models near relative equilibria
in the one-dimensional discrete nonlinear {Schr{\"o}dinger} equation.
\newblock {\em Physica D.}, 238:687--698, 2009.

\bibitem{Ro00}
G.~E. Roberts.
\newblock Linear stability in the $1+n$-gon relative equilibrium.
\newblock In J.~Delgado, editor, {\em Hamiltonian systems and celestial
mechanics. HAMSYS-98. Proceedings of the 3rd international symposium}, World
Sci. Monogr. Ser. Math. 6, pages 303--330. World Scientific, 2000.

\bibitem{Sc03}
D.S. Schmidt.
\newblock Central configurations and relative equilibria for the $n$-body
problem.
\newblock In {\em Classical and Celestial Mechanics}, pages 1--33. Princeton
Univ. Press, 2003.


\end{thebibliography}



\end{document}